\documentclass[10pt, twosides, a4paper]{article}

\usepackage{theorem,makeidx,float,calc,verbatim}
\usepackage{latexsym,amscd}
\usepackage{amsmath,amsfonts,amssymb,mathrsfs}
\usepackage{enumerate,qsymbols,euscript}

\input cyracc.def

\input xy
\xyoption{all}
\newdir{ >}{{}*!/-5pt/@{>}}
\newtheorem{theorem}{Theorem}[section]
\newtheorem{lemma}[theorem]{Lemma}
\newtheorem{proposition}[theorem]{Proposition}
\newtheorem{corollary}[theorem]{Corollary}
\theorembodyfont{\upshape}
\newtheorem{definition}[theorem]{Definition}
\newtheorem{remark}[theorem]{Remark}

\newenvironment{proof}{%
\noindent{\it Proof.}\hskip 10 pt%
}{%
{\quad} \hfill$\Box$\\%
}

\def\og{\leavevmode\raise.3ex\hbox{$\scriptscriptstyle\langle\!\langle$~}}
\def\fg{\leavevmode\raise.3ex\hbox{~$\!\scriptscriptstyle\,\rangle\!\rangle$}}
\def\endzm{\\ \hspace*{\fill} $\square$\bigskip}

\DeclareMathOperator{\Coker}{Coker}
\DeclareMathOperator{\Dec}{Dec}
\DeclareMathOperator{\des}{des}
\renewcommand{\det}{\mathrm{det}}

\DeclareMathOperator{\Fil}{\mathbf{Fil}}
\DeclareMathOperator{\HN}{HN}
\DeclareMathOperator{\Hom}{Hom}
\DeclareMathOperator{\Id}{Id}
\DeclareMathOperator{\Image}{Im}
\DeclareMathOperator{\indic}{1\!\!1}
\DeclareMathOperator{\Ker}{Ker}
\DeclareMathOperator{\obj}{obj}

\DeclareMathOperator{\pr}{pr}
\DeclareMathOperator{\rang}{rk}
\DeclareMathOperator{\Spec}{Spec}
\DeclareMathOperator{\supp}{supp}

\normalsize
 \oddsidemargin
= 10 pt \evensidemargin = 32 pt

\begin{document}
\title{Harder-Narasimhan categories}

\author{Huayi {\sc Chen}}
\maketitle

\begin{abstract}
We propose a generalization of Quillen's exact category --- arithmetic exact category and we discuss conditions on such categories under which one can establish the notion of Harder-Narasimhan filtrations and Harder-Narsimhan polygons. Furthermore, we show the functoriality of Harder-Narasimhan filtrations (indexed by $\mathbb R$), which can not be stated in the classical setting of Harder and Narasimhan's formalism.
\end{abstract}

\section{Introduction}

\hskip\parindent The notion of {\it
Harder-Narasimhan flag\footnote{In most
literature this notion is known as
``Harder-Narasimhan filtration''. However,
the so-called ``Harder-Narasimhan
filtration'' is indexed by a finite set,
therefore is in fact a flag of the vector
bundle. Here we would like to reserve the
term ``Harder-Narasimhan filtration'' for
filtration indexed by $\mathbb R$, which we
shall define later in this article.} } (or
{\it canonical flag}) of a vector bundle on a
smooth projective curve over a field was
firstly introduced by Harder and Narasimhan
\cite{Harder-Nara} to study the cohomology
groups of moduli spaces of vector bundles on
curves. Let $C$ be a smooth projective curve
on a field $k$ and $E$ be a non-zero locally
free $\mathcal O_C$-module (i.e. vector
bundle) of finite type. Harder and Narasimhan
proved that there exists a flag
\[0=E_0\subsetneq E_1\subsetneq E_2\subsetneq\cdots
\subsetneq E_n=E\] of $E$ such that
\begin{enumerate}[1)]
\item each sub-quotient $E_{i}/E_{i-1}$
($i=1,\cdots,n$) is semistable\footnote{We
say that a non-zero locally free $\mathcal
O_C$-module of finite type $F$ is semistable
if for any non-zero sub-module $F_0$ of $F$
we have $\mu(F_0)\le\mu(F)$, where the {\it
slope} $\mu$ is by definition the quotient of
the degree by the rank.} in the sense of
Mumford,
\item we have the inequality of successive slopes
\[\mu_{\max}(E):=\mu(E_1/E_0)>\mu(E_2/E_1)>\cdots>\mu(E_n/E_{n-1})=:\mu_{\min}(E).\]
\end{enumerate} The Harder-Narasimhan
polygon of $E$ is the concave function on
$[0,\rang E]$, the graph of which is the
convex hull of points of coordinate $(\rang
F,\deg(F))$, where $F$ runs over all coherent
sub-$\mathcal O_C$-modules of $E$. Its
vertexes are of coordinate $(\rang
E_i,\deg(E_i))$. The avatar of the above
constructions in Arakelov geometry was
introduced by Stuhler \cite{Stuhler76} and
Grayson \cite{Grayson76}. Similar
constructions exist also in the theory of
filtered isocristals \cite{Faltings94}.
Classically, the canonical flags have no
functoriality. Notice that already the length
of canonical flags varies when the vector
bundle $E$ changes. However, as we shall show
in this article, if we take into account the
{\it minimal slopes} of non-zero subbundles
$E_i$ in the canonical flag, which coincide
with successive slopes, i.e.
$\mu_{\min}(E_i)=\mu(E_i/E_{i-1})$, we obtain
a filtration indexed by $\mathbb R$ which we
call {\it Harder-Narasimhan filtration}. Such
construction has the functoriality.

The category of vector bundles on a
projective variety is exact in the sense of
Quillen \cite{Quillen73}. However, it is not
the case for the category of Hermitian vector
bundles on a projective arithmetic variety,
or the category of vector spaces equipped
with a filtration. We shall propose a new
notion
--- arithmetic exact category --- which
generalizes simultaneously the three cases
above. Furthermore we shall discuss the
conditions on such categories under which we
can establish the notion of semistability and
furthermore the existence of
Harder-Narasimhan filtrations. We also show
how to associate to such a filtration a Borel
probability measure on $\mathbb R$ which is a
linear combination of Dirac measures. This
construction is an important tool to study
Harder-Narasimhan polygons in the author's
forthcoming work \cite{Chen_polygon}.

We point out that the categorical approach
for studying semistability problems has been
developed in various context
by different authors, among whom we would like to
cite Bridgeland
\cite{Bridgeland03}, Lafforgue
\cite{Lafforgue97} and Rudakov \cite{Rudakov97}.

This article is organized as follows. We introduce in the second section the formalism of filtrations in an arbitrary category. In the third section, we present the arithmetic exact categories which generalizes the notion of exact categories in the sense of Quillen. We also give several examples. The fourth section is devoted to the formalism of Harder and Narasimhan on an arithmetic exact category equipped with degree and rank functions, subject to certain conditions which we shall precise (such category will be called Harder-Narasimhan category in this article). In the fifth section, we associate to each arithmetic object in a Harder-Narasimhan category a filtration indexed by $\mathbb R$, and we establish the fonctoriality of this construction. We also explain how to apply this construction to the study of Harder-Narasimhan polygons. As an application, we give a criterion of Harder-Narasimhan categories when the underlying exact category is an Abelian category. The last section contains several examples of Harder-Narasimhan categories where the arithmetic objects are classical in p-adic representation theory, algebraic geometry and Arakelov geometry respectively.

\vspace{2mm} {\bf Acknowledgement} The
results in this article is the continuation
of part of the author's doctorial thesis
supervised by J.-B. Bost to whom the author
would like to express his gratitude. The
author is also thankful to A. Chambert-Loir, B. Keller
and C. Mourougane for remarks.

\section{Filtrations in a category}

\hskip\parindent In this section we shall
introduce the notion of filtrations in a
general category and their functorial
properties. Here we are rather interested in
left continuous filtrations. However, for the
sake of completeness, and for possible
applications elsewhere, we shall also discuss
the right continuous counterpart, which is
{\bf not} dual to the left continuous case.

We fix throughout this section a non-empty
totally ordered set $I$. Let $I^*$ be the
extension of $I$ by adding a minimal element
$-\infty$. The new totally ordered set $I^*$
can be viewed as a small category. Namely,
for any pair $(i,j)$ of objects in $I^*$,
$\Hom(i,j)$ is a one point set $\{u_{ij}\}$
if $i\ge j$, and is the empty set otherwise.
The composition of morphisms is defined in
the obvious way. Notice that $-\infty$ is the
final object of $I^*$. The subset $I$ of
$I^*$ can be viewed as a full subcategory of
$I^*$.

If $i\le j$ are two elements in $I^*$, we
shall use the expression $[i,j]$ \big(resp.
$]i,j[$, $[i,j[$, $]i,j]$\big) to denote the
set $\{k\in I^*\;|\;i\le k\le j\}$ \big(resp.
$\{k\in I^*\;|\;i<k<j\}$, $\{k\in
I^*\;|\;i\le k<j \}$, $\{k\in I^*\;|\;i<k\le
j\}$\big).

\begin{definition}
Let $\mathcal C$ be a category and $X$ be an
object of $\mathcal C$. We call {\it
$I$-filtration}\index{filtration} of $X$ in
$\mathcal C$ any functor $\mathcal
F:I^*\rightarrow\mathcal C$ such that
$\mathcal F(-\infty)=X$ and that, for any
morphism $\varphi$ in $I^*$, $\mathcal
F(\varphi)$ is a monomorphism.
\end{definition}

Let $\mathcal F$ and $\mathcal G$ be two
filtrations in $\mathcal C$. We call {\it
morphism of filtrations} from $\mathcal F$ to
$\mathcal G$ any natural transformation from
$\mathcal F$ to $\mathcal G$. All filtrations
in $\mathcal C$ and all morphisms of
filtrations form a category, denoted by
$\Fil^I(\mathcal C)$. It's a full subcategory
of the category of functors from $I^*$ to
$\mathcal C$.

Let $(X,Y)$ be a pair of objects in $\mathcal
C$, $\mathcal F$ be an $I$-filtration of $X$
and $\mathcal G$ be an $I$-filtration of $Y$.
We say that a morphism $f:X\rightarrow Y$ is
{\it compatible} with the filtrations
$(\mathcal F,\mathcal G)$ if there exists a
morphism of filtrations $F:\mathcal
F\rightarrow\mathcal G$ such that
$F(-\infty)=f$. If such morphism $F$ exists,
it is unique since all canonical morphisms
$\mathcal G(i)\rightarrow Y$ are monomorphic.

We say that a filtration $\mathcal F$ is {\it
exhaustive} if
$\underrightarrow{\lim}\mathcal F|_I$ exists
and if the morphism $\varinjlim\mathcal
F|_I\rightarrow X$ defined by the system
$(\mathcal F(u_{i,-\infty}):X_i\rightarrow
X)_{i\in I}$ is an isomorphism. We say that
$\mathcal F$ is {\it separated} if
$\varprojlim\mathcal F$ exists and is an
initial object in $\mathcal C$.

If $i$ is an element in $I$, we denote by
$I_{<i}$ (resp. $I_{>i}$) the subset of $I$
consisting of all elements strictly smaller
(resp. strictly greater) than $i$. We say
that $I$ is {\it left dense} (resp. {\it
right dense}) at $i$ if $I_{<i}$ (resp.
$I_{>i}$) is non-empty and if $\sup I_{<i}=i$
(resp. $\inf I_{>i}=i$). The subsets $I_{<i}$
and $I_{>i}$ can also be viewed as full
subcategories of $I^*$.

The following two easy propositions give
criteria for $I$ to be dense (left and right
respectively) at a point $i$ in $I$.
\begin{proposition}\label{Pro:critere locale de dense}
Let $i$ be an element of $I$. The following
conditions are equivalents:
\begin{enumerate}[1)]
\item $I$ is left dense at $i$;
\item $I_{<i}$ is non-empty and
the set $]j,i[$ is non-empty for any $j<i$;
\item $I_{<i}$ is non-empty and
the set $]j,i[$ is infinite for any $j<i$.
\end{enumerate}
\end{proposition}

\begin{proposition}
Let $i$ be an element of $I$. The following
conditions are equivalents:
\begin{enumerate}[1)]
\item $I$ is right dense at $i$;
\item $I_{>i}$
is non-empty and the set $]i,j[$ is non-empty
for any $j>i$;
\item $I_{>i}$ is non-empty and the set
$]i,j[$ is infinite for any $i>j$.
\end{enumerate}
\end{proposition}

We say that a filtration $\mathcal F$ is {\it
left continuous} at $i\in I$ if $I$ is {\it
not} left dense at $i$ or if the projective
limit of the restriction of $\mathcal F$ on
$I_{<i}$  exists and the morphism $\mathcal
F(i)\rightarrow\varprojlim\mathcal
F|_{I_{<i}}$ defined by the system $(\mathcal
F(u_{ij}):\mathcal F(i)\rightarrow \mathcal
F(j))_{j<i}$ is an isomorphism. Similarly, we
say that $\mathcal F$ is {\it right
continuous} at $i\in I$ if $I$ is {\it not}
right dense at $i$ or if the inductive limit
of the restriction of $\mathcal F$ on
$I_{>i}$ exists and the morphism
$\varinjlim\mathcal
F|_{I_{>i}}\rightarrow\mathcal F(i)$ defined
by the system $(\mathcal F(u_{ji}):\mathcal
F(j)\rightarrow \mathcal F(i))_{j>i}$ is an
isomorphism. We say that a filtration
$\mathcal F$ is {\it left continuous} (resp.
{\it right continuous }) if it is left
continuous (resp. right continuous) at every
element of $I$. We denote by
$\Fil^{I,l}(\mathcal C)$ (resp.
$\Fil^{I,r}(\mathcal C)$) the full
subcategory of $\Fil^I(\mathcal C)$ formed by
all left continuous (resp. right continuous)
filtrations in $\mathcal C$.

Given an arbitrary filtration $\mathcal F$,
we want to construct a left continuous
filtration which is ``closest'' to the
original one. The best candidate is of course
the filtration $\mathcal F^l$ such that
\[\mathcal
F^l(i)=\begin{cases}\varprojlim_{k<i}\mathcal
F(k),&I\text{ is left dense at }i,\\
\mathcal F(i),&\text{otherwise}.\end{cases}\]
However, this filtration is well defined only
when all projective limits
$\varprojlim_{k<i}\mathcal F(k)$ exist for
any $i\in I$ where $I$ is left dense.
Therefore, under the following supplementary
condition ({\bf M}) for the category
$\mathcal C$:
\begin{quote}{\it any non-empty totally ordered system
of monomorphisms in $\mathcal C$ has a
projective limit,}\end{quote} for any
filtration $\mathcal F$ in $\mathcal C$, the
filtration $\mathcal F^l$ exists.
Furthermore, $\mathcal F\longmapsto\mathcal
F^l$ is a functor, which is left adjoint to
the forgetful functor from
$\mathbf{Fil}^{I,l}(\mathcal C)$ to
$\mathbf{Fil}^I(\mathcal C)$.

Similarly, given an arbitrary filtration
$\mathcal F$ of an object in $\mathcal C$, if
for any $i\in I$ where $I$ is right dense,
the inductive limit of the system $(\mathcal
F(j))_{j>i}$ exists, and the canonical
morphism $\varinjlim_{j>i}\mathcal
F(j)\rightarrow X$ defined by the system
$(\mathcal F(u_{j,-\infty}):\mathcal
F(j)\rightarrow X)_{j>i}$ is monomorphic,
then the filtration $\mathcal F^r$ such that
\[\mathcal F^r(i)=\begin{cases}
\varinjlim_{j>i}\mathcal F(j),&I\text{ is
right dense at }i,\\
\mathcal F(i),&\text{otherwise},
\end{cases}\]
is right continuous. Therefore, if the
following condition ($\mathbf{M}^*$) is
fulfilled for the category $\mathcal C$:
\begin{quote}\it
any non-empty totally ordered system
$(\xymatrix{\relax
X_i\ar[r]^-{\alpha_i}&X})_{i\in J}$ of
subobjects of an object $X$ in $\mathcal C$
has an inductive limit, and the canonical
morphism $\displaystyle\varinjlim
X_i\rightarrow X$ induced by
$(\alpha_i)_{i\in J}$ is monomorphic,
\end{quote}
then for any filtration $\mathcal F$ in
$\mathcal C$, the filtration $\mathcal F^r$
exists, and $\mathcal F\longmapsto\mathcal
F^r$ is a functor, which is right adjoint to
the forgetful functor from
$\mathbf{Fil}^{I,r}(\mathcal C)$ to
$\mathbf{Fil}^I(\mathcal C)$.

Let $X$ be an object in $\mathcal C$. All
$I$-filtrations of $X$ and all morphisms of
filtrations equalling to $\Id_X$ at $-\infty$
form a category, denoted by $\Fil_{X}^I$. We
denote by $\Fil_{X}^{I,l}$ (resp.
$\Fil_{X}^{I,r}$) the full subcategory of
$\Fil_{X}^I$ consisting of all left
continuous (resp. right continuous)
filtrations of $X$. The category $\Fil_{X}^I$
has a final object $C_X$ which sends all
$i\in I^*$ to $X$ and all morphisms in $I^*$
to $\Id_X$. We call it the {\it trivial}
filtration of $X$. If the condition ({\bf M})
is verified for the category $\mathcal C $,
the restriction of the functor $\mathcal
F\longmapsto\mathcal F^l$ on
$\mathbf{Fil}_{X}^I$ is a functor from
$\mathbf{Fil}_{X}^I$ to
$\mathbf{Fil}_{X}^{I,l}$, which is left
adjoint to the forgetful functor
$\mathbf{Fil}_{X}^{I,l}\rightarrow
\mathbf{Fil}_{X}^I$. Similarly, if the
condition ($\mathbf{M}^*$) is verified for
the category $\mathcal C$, the restriction of
the functor $\mathcal F\longmapsto\mathcal
F^r$ on $\mathbf{Fil}_{X}^I$ gives a functor
from $\mathbf{Fil}_{X}^I$ to
$\mathbf{Fil}_{X}^{I,r}$, which is right
adjoint to the forgetful functor
$\mathbf{Fil}_{X}^{I,r}\rightarrow
\mathbf{Fil}_{X}^I$.

In the following, we shall discuss functorial
constructions of filtrations. Namely, given a
morphism $f:X\rightarrow Y$ in a category
$\mathcal C$ and a filtration of $X$ or $Y$,
we shall explain how to construct a
``natural'' filtration of the other.

Suppose that $f:X\rightarrow Y$ is a morphism
in $\mathcal C$ and $\mathcal G$ is an
$I$-filtration of $Y$. If the fiber product
in the functor category
$\mathbf{Fun}(I^*,\mathcal C)$, defined by
$f^*\mathcal G:=\mathcal G\times_{C_Y}C_X$,
exists, where $C_X$ (resp. $C_Y$) is the
trivial filtration of $X$ (resp. $Y$), then
the functor $f^*\mathcal G$ is a filtration
of $X$. We call it the {\it inverse image} of
$\mathcal G$ by the morphism $f$. The
canonical projection $P$ from $f^*\mathcal G$
to $\mathcal G$ gives a morphism of
filtrations in $\Fil^I(\mathcal C)$ such that
$P(-\infty)=f$. In other words, the morphism
$f$ is compatible with the filtrations
$(f^*\mathcal G,\mathcal G)$. Since the fiber
product commutes to projective limits, if
$\mathcal G$ is left continuous at a point
$i\in I$, then also is $f^*\mathcal G$.

If in the category $\mathcal C$, all fiber
products exist\footnote{In this case, for any
small category $\mathcal D$, the category of
functors from $\mathcal D$ to $\mathcal C$
supports fiber products. In particular, all
fiber products in the category
$\mathbf{Fil}^I(\mathcal C)$ exist.}, then
for any morphism $f:X\rightarrow Y$ in
$\mathcal C$ and any filtration $\mathcal G$
of $Y$, the inverse image of $\mathcal G$ by
$f$ exists, and $f^*$ is a functor from
$\mathbf{Fil}_{Y}^I$ to $\mathbf{Fil}_{X}^I$
which sends $\mathbf{Fil}_{Y}^{I,l}$ to
$\mathbf{Fil}_{X}^{I,l}$.

Let $\mathcal C$ be a category and
$f:X\rightarrow Y$ be a morphism in $\mathcal
C$. We call {\it admissible decomposition} of
$f$ any triplet $(Z,u,v)$ such that:
\begin{enumerate}[1)]
\item $Z$ is an object of $\mathcal C$,
\item $u:X\rightarrow Z$ is a morphism in $\mathcal
C$ and $v:Z\rightarrow Y$ is a monomorphism
in $\mathcal C$ such that $f=vu$.
\end{enumerate}
If $(Z,u,v)$ and $(Z',u',v')$ are two
admissible decompositions of $f$, we call
{\it morphism of admissible decompositions}
from $(Z,u,v)$ to $(Z',u',v')$ any morphism
$\varphi:Z\rightarrow Z'$ such that $\varphi
u=u'$ and that $v=v'\varphi$.
\[\xymatrix{&Z\ar'[d][dd]|>>>>>>>{\varphi}\ar[rd]^v\\
X\ar[rd]_{u'}\ar[ru]^u\ar[rr]|<<<<<<<f&&Y\\
&Z'\ar[ru]_{v'}}\] All admissible
decompositions and their morphisms form a
category, denoted by $\Dec(f)$. If the
category $\Dec(f)$ has an initial object
$(Z_0,u_0,v_0)$, we say that $f$ has an {\it
image}. The monomorphism $v_0:Z_0\rightarrow
Y$ is called an {\it image} of $f$, or an
{\it image} of $X$ in $Y$ by the morphism
$f$, denoted by $\Image f$.

Suppose that $f:X\rightarrow Y$ is a morphism
in $\mathcal C$ and that $\mathcal F$ is a
filtration of $X$. If for any $i\in I$, the
morphism $f\circ\mathcal
F(u_{i,-\infty}):\mathcal F(i)\rightarrow Y$
has an image, then we can define a filtration
$f_{\flat}\mathcal F$ of $Y$, which
associates to each $i\in I$ the subobject
$\Image(f\circ\mathcal F(u_{i,-\infty}))$ of
$Y$. This filtration is called the {\it weak
direct image} of $\mathcal F$ by the morphism
$f$. If furthermore the filtration
$f_*\mathcal F:=(f_{\flat}\mathcal F)^l$ is
well defined, we called it the {\it strong
direct image} by $f$. Notice that for any
filtration $\mathcal F$ of $X$, the morphism
$f$ is compatible with filtrations $(\mathcal
F,f_{\flat}\mathcal F)$ and $(\mathcal
F,f_*\mathcal F)$ (if $f_{\flat}\mathcal F$
and $f_*\mathcal F$ are well defined).
Moreover, if any morphism in $\mathcal C$ has
an image, then $f_{\flat}$ is a functor from
$\Fil_{X}^I$ to $\Fil_{Y}^I$. If in addition
the condition ({\bf M}) is fulfilled for the
category $\mathcal C$, $f_*$ is a functor
from $\Fil_{X}^I$ to $\Fil_{Y}^{I,l}$.

\begin{proposition}
Let $\mathcal C$ be a category which supports
fiber products and such that any morphism in
it admits an image. If $f:X\rightarrow Y$ is
a morphism in $\mathcal C$, then the functor
$f^*:\Fil_{Y}^I\rightarrow \Fil_{X}^I$ is
right adjoint to the functor $f_{\flat}$.
\end{proposition}
\begin{proof} Let $\mathcal F$ be a
filtration of $Y$, $\mathcal G$ be a
filtration of $X$ and $\tau:\mathcal
G\rightarrow f^*\mathcal F$ be a morphism.
For any $i\in I$ let $\varphi_i:\mathcal
F(i)\rightarrow Y$ and $\psi_i:\mathcal
G(i)\rightarrow X$ be canonical morphisms,
and let $(f_{\flat}\mathcal G(i),u_i,v_i)$ be
an image of $\mathcal G(i)$ by the morphism
$f\psi_i$. Since the morphism
$\varphi_i:\mathcal F(i)\rightarrow Y$ is
monomorphic, there exists a unique morphism
$\eta_i$ from $f_{\flat}\mathcal G(i)$ to
$\mathcal F(i)$ such that
$\varphi_i\eta_i=v_i$ and that
$\eta_iu_i=\pr_1\tau(i)$.
\[\xymatrix{& f^*\mathcal
F(i)\ar[rr]^{\pr_1}\ar'[d][dd]&&\mathcal
F(i)\ar[dd]^{\varphi_i}\\
\mathcal G(i)\ar[ru]^{\tau(i)}
\ar[rr]|<<<<<<<{\,u_i}\ar[rd]_{\psi_i}
&&f_{\flat}\mathcal
G(i)\ar[ru]^{\eta_i}\ar[rd]_{v_i}\\
&X\ar[rr]_f&&Y }\] Hence we have a functorial
bijection $\Hom_{\Fil_{X}^I}(\mathcal
G,f^*\mathcal
F)\stackrel{\sim}{\longrightarrow}\Hom_{
\Fil_{Y}^I}(f_{\flat}\mathcal G,\mathcal F)$.
\end{proof}

\begin{corollary}
With the notations of the previous
proposition, if we suppose in addition that
the condition $(\mathbf{M})$ is verified for
the category $\mathcal C$, then for any
morphism $f:X\rightarrow Y$ in $\mathcal C$,
the functor
$f^*:\Fil_{Y}^{I,l}\rightarrow\Fil_{X}^I$ is
right adjoint to the functor $f_*$.
\end{corollary}
\begin{proof} For any filtration
$\mathcal F$ of $X$ and any left continuous
filtration $\mathcal G$ of $Y$, we have the
following functorial bijections
\[\Hom_{\Fil_{X}^I}(\mathcal F,f^*\mathcal G)\stackrel{\sim}
{\longrightarrow}\Hom_{\Fil_{Y}^I}(f_{\flat}\mathcal
F,\mathcal
G)\stackrel{\sim}{\longrightarrow}\Hom_{
\Fil_{Y}^{I,l}}(f_*\mathcal F,\mathcal G).\]
\end{proof}

In the last part of the section, we shall
discuss a special type of filtrations, namely
filtrations of finite length, which are
important for later sections.

Let $\mathcal C$ be a category. We say that a
filtration $\mathcal F$ of $X\in\obj\mathcal
C$ is {\it of finite length} if there exists
a finite subset $I_0$ of $I$ such that, for
any $i>j$ satisfying
$I_0\cap[j,i]=\emptyset$, the morphism
$\mathcal F(u_{ij})$ is isomorphic. The
subset $I_0$ of $I$ is called a {\it jumping
set} of $\mathcal F$. We may have different
choices of jumping set. In fact, if $I_1$ is
an arbitrary finite subset of $I$ and if
$I_0$ is a jumping set of $\mathcal F$, then
$I_0\cup I_1$ is also a jumping set of
$\mathcal F$. However, the intersection of all jumping sets of $\mathcal F$ is itself a jumping set, called the {\it minimal jumping set} of $\mathcal F$.

Let $f:X\rightarrow Y$ be a morphism in
$\mathcal C$. If $\mathcal G$ is a filtration
of finite length of $Y$ such that
$f^*\mathcal G$ is well defined, then the
filtration $f^*\mathcal G$ is also of finite
length since the fibre product preserves
isomorphisms.

Let $\mathcal C$ be a category, $X$ be an
object in $\mathcal C$ and $\mathcal F$ be an
$I$-filtration of $X$. We say that $\mathcal
F$ is {\it left locally constant} at $i\in I$
if $I$ is not left dense at $i$ or if there
exists $j<i$ such that $\mathcal F(u_{ij})$
is an isomorphism, or equivalently $\mathcal
F(u_{ik})$ is an isomorphism for any
$k\in[j,i[$. Similarly, we say that $\mathcal
F$ is {\it right locally constant} at $i$ if
$I$ is not right dense at $i$ or if there
exists $j>i$ such that $\mathcal F(u_{ji})$
is an isomorphism, or equivalently, $\mathcal
F(u_{ki})$ is an isomorphism for any
$k\in]i,j]$. We say that the filtration
$\mathcal F$ is {\it left locally constant}
(resp. {\it right locally constant}) if it is
left locally constant (resp. right locally
constant) at any point $i\in I$.

\begin{proposition}
Let $\mathcal C$ be a category, $X$ be an
object in $\mathcal C$, $\mathcal F$ be a
filtration of finite length of $X$, and $I_0$
be a jumping set of $\mathcal F$. For any
$i\in I\setminus I_0$, the filtration
$\mathcal F$ is left and right locally
constant at $i$.
\end{proposition}
\begin{proof}
Let $i\in I\setminus I_0$ be an element where
$I$ is left dense. Since $I_0$ is a finite
set, also is $I_{<i}\cap I_0$. Let
$j_0=\max(I_{<i}\cap I_0)$. We have $j_0<i$,
therefore the set $]j_0,i[$ is non-empty since $I$ is left dense at $i$. Choose an arbitrary element $j\in
]j_0,i[$. We have $[j,i]\cap I_0=\emptyset$,
so $\mathcal F(u_{i,j})$ is an isomorphism.
Therefore, $\mathcal F$ is left locally
constant at $i$. The proof for the fact that
$\mathcal F$ is right locally constant at $i$
is similar.
\end{proof}

\begin{proposition}
If the filtration $\mathcal F$ is left
locally constant (resp. right locally
constant), then it is left continuous (resp.
right continuous). The converse is true when
the filtration $\mathcal F$ is of finite
length.
\end{proposition}
\begin{proof}
``$\Longrightarrow$'' is trivial.

``$\Longleftarrow$'': Suppose that $\mathcal
F$ is a left continuous filtration of finite
length. Let $I_0$ be a jumping set of
$\mathcal F$. If $I$ is left dense at $i$,
there then exists an element $j<i$ in $I$
such that $[j,i[\cap I_0=\emptyset$. Since
$\mathcal F$ is left continuous at $i$,
$\mathcal F(i)$ is the projective limit of a
totally ordered system of isomorphisms.
Therefore $\mathcal F(u_{ij})$ is an
isomorphism. The proof of the other assertion
is the same.
\end{proof}

\begin{corollary}
Let $\mathcal C$ be a category and $\mathcal
F$ be a filtration of finite length in
$\mathcal C$. If $\mathcal F^l$ (resp.
$\mathcal F^r$) is well defined, then it is
also of finite length.
\end{corollary}
\begin{proof}
Let $I_0$ be a jumping set of $\mathcal F$.
We know that the filtration $\mathcal F$ is
left continuous outside $I_0$, hence for any
$i\in I\setminus I_0$ we have $\mathcal
F^l(i)=\mathcal F(i)$, and if $[j,i]\subset
I$ doesn't encounter $I_0$, then $\mathcal
F^l(u_{ij})=\mathcal F(u_{ij})$. Therefore,
$\mathcal F^l$ is of finite length, and $I_0$
is a jumping set of $\mathcal F^l$. The proof
for the other assertion is similar.
\end{proof}

\begin{corollary}
Let $\mathcal C$ be a category,
$f:X\rightarrow Y$ be a morphism in $\mathcal
C$ and $\mathcal F$ be a filtration of finite
length of $X$. If $f_{\flat}(\mathcal F)$
(resp. $f_*(\mathcal F)$) is well defined,
then it is also of finite length.
\end{corollary}

Let $\mathcal C$ be a category and $\mathcal
F$ be an $I$-filtration of an object $X$ in
$\mathcal C$ which is of finite length. The
filtration $\mathcal F$ is exhaustive if and
only if there exists $i_1\in I$ such that
$\mathcal F(u_{ji})$ is isomorphism for all
$i\le j\le i_1$ in $I^*$. Suppose that
$\mathcal C$ has an initial object, then the
filtration $\mathcal F$ is separated if and
only if there exists $i_2\in I$ such that
$\mathcal F(i)$ is an initial object for any
$i\ge i_2$.

Let $\mathcal C$ and $\mathcal D$ be two
categories and $F:\mathcal
C\rightarrow\mathcal D$ be a functor. If $F$
sends any monomorphism in $\mathcal C$ to a
monomorphism in $\mathcal D$, then $F$
induces a functor $\widetilde
F:\Fil^I(\mathcal
C)\rightarrow\Fil^I(\mathcal D )$. If $X$ is
an object of $\mathcal C$ and if $\mathcal F$
is a filtration of $X$, $\widetilde
F(\mathcal F)$ is called the filtration {\it
induced} from $\mathcal F$ by the functor
$F$. The following assertions can be deduced
immediately from definition:
\begin{enumerate}[i)]
\item If $\mathcal F$ is left locally constant
(resp. right locally constant, of finite
length), then $\widetilde F(\mathcal F)$ is
left locally constant (resp. right locally
constant, of finite length).
\item If
$\mathcal F$ is of finite length and
exhaustive, then also is $\widetilde
F(\mathcal F)$.
\item Suppose that $\mathcal C$ and
$\mathcal D$ have initial objects and that
$F$ preserves initial objects. If $\mathcal
F$ is separated and of finite length, the
also is $\widetilde F(\mathcal F)$.
\end{enumerate}

\section{Arithmetic exact categories}

\hskip\parindent The notion of exact
categories is defined by Quillen
\cite{Quillen73}. It is a generalization of
Abelian categories. For example, the category
of all locally free sheaves on a smooth
projective curve is an exact category, but it
is not an Abelian category. Furthermore,
there are natural categories which fail to be
exact, but look alike. A typical example is
the category of Hermitian vector spaces and
linear applications of norm $\le 1$. Another
example is the category of finite dimensional
filtered vector spaces over a field and
linear applications which are compatible with
filtrations. A common characteristic of such
categories is that any object in such a
category can be described as an object in an
exact category equipped with certain
additional structure. In the first example,
it is a finite dimensional complex vector
space equipped with a Hermitian metric; and
in the second one, it is a finite dimensional
vector space equipped with a filtration.

In the following we shall formalize the above
observation by a new notion --- {\it
arithmetic exact category} --- by proposing
some axioms and we shall provide several
examples. Let us begin by recalling the exact
categories in the sense of Quillen. Let
$\mathcal C$ be an essentially
small category and
let $\mathcal E$ be a class of diagrams in
$\mathcal C$ of the form
\[\xymatrix{0\ar[r]&X'\ar[r]&X\ar[r]&X''\ar[r]&0}.\]
If $\xymatrix{\relax 0\ar[r]&X'\ar[r]^-f&X
\ar[r]^-g&X''\ar[r]&0}$ is a diagram in
$\mathcal E$, we say that $f$ is an {\it
admissible monomorphism} and that $g$ is an
{\it admissible epimorphism}. We shall use
the symbol ``$\xymatrix{{}\ar@{
>->}[r]&}$'' to denote an admissible
monomorphism, and
``$\xymatrix{{}\ar@{->>}[r]&}$'' for an
admissible epimorphism.

If $\mathcal
F:\xymatrix{0\ar[r]&X'\ar[r]&X\ar[r]&X''\ar[r]&0}$
and $\mathcal
G:\xymatrix{0\ar[r]&Y'\ar[r]&Y\ar[r]&Y''\ar[r]&0}$
are two diagrams of morphisms in $\mathcal
C$, we call {\it morphism} from $\mathcal F$
to $\mathcal G$ any commutative diagram
\[\xymatrix{\relax 0\ar@{}[d]_-{{\displaystyle (\Phi):\quad}}\ar[r]&X'\ar[r]\ar[d]_{\varphi'}&
X\ar[r]\ar[d]_{\varphi}&X''\ar[r]\ar[d]^{\varphi''}&0\\
0\ar[r]&Y'\ar[r]&Y\ar[r]&Y''\ar[r]&0.}
\] We say that $(\Phi)$ is an {\it
isomorphism} if $\varphi'$, $\varphi$ and
$\varphi''$ are all isomorphisms in $\mathcal
C$.

\begin{definition}[Quillen]\label{Def:catetgorie exacte}
We say that $(\mathcal C,\mathcal E)$ is an
{\it exact category} if the following axioms
are verified:
\begin{enumerate}[($\mathbf{Ex}$1)]
\item \label{Axiom:ex epi et mono dans une suite exacte}For ay diagram
\[\xymatrix{0\ar[r]&X'\ar[r]^\varphi&X\ar[r]^\psi
&X''\ar[r]&0}\] in $\mathcal E$, $\varphi$ is
a {\it kernel} of $\psi$ and $\psi$ is a {\it
cokernel} of $\varphi$.
\item \label{Axiom:ex somme direct}
If $X$ and $Y$ are two objects in $\mathcal
C$, then the diagram
\[\xymatrix{\relax 0\ar[r]&X\ar[r]^-{(\Id,0)}&X\amalg Y\ar[r]^-{\pr_2}&Y\ar[r]&0}\]
is in $\mathcal E$.
\item \label{Axiom:ex E invariant par iso} Any diagram which is isomorphe to a diagram in $\mathcal E$ lies also in $\mathcal
E$.
\item \label{Axiom:ex composition invariant epi et mono} If $f:X\rightarrow Y$ and
$g:Y\rightarrow Z$ are two admissible
monomorphisms (resp. admissible
epimorphismes), then also is $gf$.
\item \label{Axiom:ex mono stable par coproduit}
For any admissible monomorphism
$f:X'\rightarrowtail X$ and any morphism
$u:X'\rightarrow Y$ in $\mathcal C$, the
fiber coproduct of $f$ and $u$ exists.
Furthermore, if the diagram
\[\xymatrix{X'\ar@{ >->}[r]^f\ar[d]_u&X\ar[d]^v\\
Y\ar[r]_g&Z}\] is cocartesian, then $g$ is an
admissible monomorphism.
\item \label{Axiom: ex epi stable par probduit} For any admissible epimorphism
$f:X\twoheadrightarrow X''$ and any morphism
$u:Y\rightarrow X''$ in $\mathcal C$, the
fiber product of $f$ and $u$ exists.
Furthermore, if the diagram
\[\xymatrix{Z\ar[r]^g\ar[d]_v&Y\ar[d]^u\\ X\ar@{->>}[r]_f&X''}\]
is cartesian, then $g$ is an admissible
epimorphism.
\item \label{Axiom: ex composition ep mon implie un epi mono}For any morphism $f:X\rightarrow Y$ in $\mathcal C$
having a kernel (resp. cokernel), if there
exists an morphism $g:Z\rightarrow X$ (resp.
$g:Y\rightarrow Z$) such that $fg$ (resp.
$gf$) is an admissible epimorphism (resp.
admissible monomorphism), then also is $f$
itself.
\end{enumerate}
\end{definition}

Keller \cite{Keller90} has shown that the axiom ({\bf Ex}\ref{Axiom: ex composition ep mon implie un epi mono}) is actually a consequence of the other axioms.

If $f:X\rightarrow Y$ is an admissible
monomorphism, by the axiom
($\mathbf{Ex}$\ref{Axiom:ex epi et mono dans
une suite exacte}), the morphism $f$ admits a
cokernel which we shall note $Y/X$. The pair
$(X,f)$ is called an {\it admissible
subobject} of $Y$.

After \cite{Quillen73}, if $\mathcal C$ is an
Abelian category and if $\mathcal E$ is the
class of all exact sequences in $\mathcal C$,
then $(\mathcal C,\mathcal E)$ is an exact
category. Furthermore, any exact category can
be naturally embedded (through the additive
version of Yoneda's functor) into an Abelian
category.

The following result is important for the
axiom ({\bf A}\ref{Axiome:A recollement}) in
Definition \ref{Def:categorie exacte
arithematque} below.

\begin{proposition}\label{Pro:decomposition canonique}
Let $(\mathcal C,\mathcal E)$ be an exact
category and $f:X\rightarrow Y$ be a morphism
in $\mathcal C$.
\begin{enumerate}[1)]
\item The diagram
\[\xymatrix{\relax 0\ar[r]&X\ar[r]^-{(\Id_X,f)}&X\amalg Y\ar[rr]^-{f\circ\pr_1-\pr_2}&&Y\ar[r]&0}\]
is in $\mathcal E$.
\item The morphism $f$ factorizes as $f=\pr_2\circ(\Id_X,f)$. Furthermore, the second projection $\pr_2:X\amalg Y\rightarrow Y$ is an admissible epimorphism and $(\Id_X,f):X\rightarrow X\amalg Y$ is an admissible monomorphism.
\end{enumerate}
\end{proposition}
\begin{proof}
Let $Z=X\amalg Y$. Consider the morphisms
$u=(\Id_X,f):X\rightarrow Z$ and
$v=\pr_2:Z\rightarrow Y$. Clearly we have
$vu=f$. Moreover, after the axiom ({\bf
Ex}\ref{Axiom:ex somme direct}), $v$ is an
admissible epimorphism. Therefore it suffices
to verify that $u$ is an admissible
monomorphism. Consider the morphism
$w=f\circ\pr_1-\pr_2:Z\rightarrow Y$. We
shall prove that $w$ is the cokernel of $u$.
First we have $wu=0$. Furthermore, any
morphism $\alpha:Z\rightarrow S$ can be
written in the form
$\alpha=\alpha_1\circ\pr_1-\alpha_2\circ\pr_2:Z\rightarrow
S$, where $\alpha_1\in\Hom(X,S)$ and
$\alpha_2\in\Hom(Y,S)$. If $\alpha u=0$, we
have $\alpha_2f=\alpha_1$, i.e., the diagram
\[\xymatrix{\relax X\ar[r]^-u&X\amalg Y\ar[r]^-w\ar[rd]_-\alpha&Y\ar@{.>}[d]^-{\alpha_2}\\
&&S}\] is commutative. Finally, if
$\beta:Y\rightarrow S$ satisfies $\beta
w=\beta f-\beta=\alpha$, then
$\beta=\alpha_2$. Therefore, we have proved
that $u$ has a cokernel.

Since the composition of moprhisms
$\xymatrix{X\ar[r]^u&Z\ar[r]^{\pr_1}&X}$ is
the identity morphism $\Id_X$, which is an
admissible monomorphism, we know, thanks to
axiom ({\bf Ex}\ref{Axiom: ex composition ep
mon implie un epi mono}), that $u$ is also an
admissible monomorphism.
\end{proof}

We now introduce the notion of arithmetic
exact categories. As explained above, an
arithmetic exact category is an exact
category where each object is equipped with a
set of ``arithmetic structures'', subject to
some compatibility conditions (axioms ({\bf
A}\ref{Axiome:A A(0) est singleton}) ---
({\bf A}\ref{Axiome:A carre cartesien})).
Finally, the axiom ({\bf A}\ref{Axiome:A
recollement}) shall be used to describe
morphisms compatible with arithmetic
structures.

\begin{definition}
\label{Def:categorie exacte arithematque} Let
$(\mathcal C,\mathcal E)$ be an exact
category. We call {\it arithmetic structure}
on $(\mathcal C,\mathcal E)$ the following
data:
\begin{enumerate}[1)]
\item a mapping $A$ from $\obj\mathcal C$ to the class of
sets,
\item for any admissible monomorphism $f:X\rightarrow
Y$, a mapping $f^*:A(Y)\rightarrow A(X)$,
\item for any admissible epimorphism $g:X\rightarrow
Y$, a mapping $g_*:A(X)\rightarrow A(Y)$,
\end{enumerate}
subject to the following axioms:
\begin{enumerate}[({\bf A}1)]
\item\label{Axiome:A A(0) est singleton} $A(0)$ is a one-point set,
\item\label{Axiome:A fontorialite limage reciproque} if $\xymatrix{X\ar@{ >->}[r]^{i}&Y\ar@{ >->}[r]^{j}&Z}$ are admissible monomorphisme, we have $(ji)^*=i^*j^*$,
\item\label{Axiome:A fonctorialite limage directe} if  $\xymatrix{X\ar@{->>}[r]^p&Y\ar@{->>}[r]^q&Z}$ are admissible epimorphisms, we have $(qp)_*=q_*p_*$,
\item\label{Axiome:A identigy} for any object $X$ of $\mathcal C$, $\Id_X^*=\Id_{X*}=\Id_{A(X)}$,
\item \label{Axiome:A iosmorphisme} if $f:X\rightarrow Y$ is an isomorphism, we have
$f^*f_*=\Id_{A(X)}$ and $f_*f^*=\Id_{A(Y)}$,
\item\label{Axiome:A carre cartesien} for any cartesian or\footnote{Here we can prove that the square is actually cartesian {\bf and} cocartesian.} cocartesian square
\begin{equation}\label{Equ:A carre cartesien}\xymatrix{\relax X\ar@{
>->}[r]^-u\ar@{->>}[d]_p&
Y\ar@{->>}[d]^-q\\
Z\ar@{ >->}[r]_-v&W}\end{equation} in
$\mathcal C$, where $u$ and $v$ (resp. $p$
and $q$) are admissible monomorphisms (resp.
admissible epimorphisms), we have
$v^*q_*=p_*u^*$,
\item\label{Axiome:A recollement}
if $\xymatrix{X\ar@{->>}[r]^u&Y\ar@{
>->}[r]^v&Z}$ is a diagram in $\mathcal C$
where $u$ (resp. $v$) is an admissible
epimorphism (resp. admissible monomorphism)
and if $(h_X,h_Z)\in A(X)\times A(Z)$
satisfies $u_*(h_X)=v^*(h_Z)$, then there
exists $h\in A(X\amalg Z)$ such that
$(\Id,vu)^*(h)=h_X$ and that
$\pr_{2*}(h)=h_Z$.
\end{enumerate}
The triplet $(\mathcal C,\mathcal E,A)$ is
called an {\it arithmetic exact category}.
For any object $X$ of $\mathcal C$, we call
{\it arithmetic structure} on $X$ any element
$h$ in $A(X)$. The pair $(X,h)$ is called an
{\it arithmetic object} in $(\mathcal
C,\mathcal E,A)$. If $p:X\rightarrow Z$ is an
admissible epimorphism, $p_*(h)$ is called
the {\it quotient arithmetic structure} on
$Z$. If $i:Y\rightarrow X$ is an admissible
monomorphism, $i^*h$ is called the {\it
induced arithmetic structure} on $Y$.
$(Z,p_*(h))$ is called an {\it arithmetic
quotient} and $(Y,i^*(h)$ is called an {\it
arithmetic subobject} of $(X,h)$.
\end{definition}

Let $(\mathcal C,\mathcal E)$ be an exact
category. If for any object $X$ of $\mathcal
C$, we denote by $A(X)$ a one point set, and
we define induced and quotient arithmetic
structure in the obvious way, then $(\mathcal
C,\mathcal E,A )$ becomes an arithmetic exact
category. The arithmetic structure $A$ is
called the {\it trivial arithmetic structure}
on the exact category $(\mathcal C,\mathcal
E)$. Therefore, exact categories can be
viewed as {\it trivial} arithmetic exact
categories.

Let $(\mathcal C,\mathcal E,A)$ be an
arithmetic exact category. If $(X',h')$ and
$(X'',h'')$ are two arithmetic objects in
$(\mathcal C,\mathcal E,A)$, we say that a
morphism $f:X'\rightarrow X''$ in $\mathcal
C$ is {\it compatible} with arithmetic
structures if there exists an arithmetic
object $(X,h)$, an admissible monomorphism
$u:X'\rightarrow X$ and an admissible
epimorphism $v:X\rightarrow X''$ such that
$h'=u^*(h)$ and that $h''=v_*(h)$.

From the definition of morphisms compatible
with arithmetic structures, we obtain the
following assertions:
\begin{enumerate}[1)]
\item If $(X_1,h_1)$ and $(X_2,h_2)$
are two arithmetic objects and if
$f:X_1\rightarrow X_2$ is an admissible
monomorphism (resp. admissible epimorphism)
such that $f^*h_2=h_1$ (resp. $f_*h_1=h_2$),
then $f$ is compatible with arithmetic
structures.
\item If $(X_1,h_1)$ and $(X_2,h_2)$
are two arithmetic objects and if
$f:X_1\rightarrow X_2$ is the zero morphism,
then $f$ is compatible with arithmetic
structures.
\item The composition of two morphisms compatible
with arithmetic structure is also compatible
with arithmetic structure. This is a
consequence of the axiom ({\bf
A}\ref{Axiome:A recollement}).
\end{enumerate}

Let $(\mathcal C,\mathcal E,A)$ be an
arithmetic exact category. After the argument
3) above, all arithmetic objects in
$(\mathcal C,\mathcal E,A )$ and morphisms
compatible with arithmetic structures form a
category which we shall denote by $\mathcal
C_A$. In the following, in order to simplify
the notations, we shall use the symbol
$\overline X$ to denote an arithmetic object
$(X,h)$ if there is no ambiguity on the
arithmetic structure $h$.

Let $(\mathcal C,\mathcal E)$ be an exact
category. Suppose that $(A_i)_{i\in I}$ is a
family of arithmetic structure on $(\mathcal
C,\mathcal E )$. For any object $X$ in
$\mathcal C$, let $A(X)=\prod_{i\in
I}A_i(X)$. Suppose that $h=(h_i)_{i\in I}$ is
an element in $A(X)$. For any admissible
monomorphism $u:Y\rightarrow X$, we define
$u^*h:=(u^*h_i)_{i\in I}\in A(Y)$; for any
admissible epimorphism $\pi:X\rightarrow Z$,
we define $\pi_*h=(\pi_*h_i)_{i\in I}\in
A(Z)$. Then it is not hard to show that $A$
is an arithmetic structure on $(\mathcal
C,\mathcal E)$. We say that $A$ is the {\it
product arithmetic structure} of $(A_i)_{i\in
I}$, denoted by $\prod_{i\in I}A_i$.

We now give some examples of arithmetic exact
categories.

\subsection*{Hermitian
spaces}

\hskip\parindent Let $\mathbf{Vec}_{\mathbb
C}$ be the category of finite dimensional
vector spaces over $\mathbb C$. It is an
Abelian category. Let $\mathcal E$ be the
class of short exact sequences of finite
dimensional vector spaces.  For any finite
dimensional $\mathbb C$-vector space $E$ over
$\mathbb C$, let $A(E)$ be the set of all
Hermitian metrics on $E$. Suppose that $h$ is
a Hermitian metric on $E$. If
$i:E_0\rightarrow E$ is a subspace of $E$, we
denote by $i^*(h)$ the induced metric on
$E_0$. If $\pi:E\rightarrow F$ is a quotient
space of $E$, we denote by $\pi_*(h)$ the
quotient metric on $F$. Then
$(\mathbf{Vec}_{\mathbb C},\mathcal E,A)$ is
an arithmetic exact category. In fact, the
axioms ({\bf A}\ref{Axiome:A A(0) est
singleton}) --- ({\bf A}\ref{Axiome:A carre
cartesien}) are easily verified. The
verification of the axiom ({\bf
A}\ref{Axiome:A recollement}) relies on the
following proposition.
\begin{proposition}\label{Pro:inverse of projection - inclusion}
Let $E$, $F_0$ and $F$ be Hermitian spaces
such that $F_0$ is a quotient Hermitian space
of $E$ and a Hermitian subspace of $F$. We
denote by $\pi:E\rightarrow F_0$ the
projection of $E$ onto $F_0$ and by
$i:F_0\rightarrow F$ the inclusion and we
note $\varphi=i\pi$. Then there exists a
Hermitian metric on $E\oplus F$ such that in
the diagram
\[\xymatrix{\relax 0\ar[r]&E\ar[r]^-{(\Id,\varphi)}&E\oplus
F \ar[r]^-{\pr_2}&F\ar[r]&0},\]
$(\Id,\varphi):E\rightarrow E\oplus F$ is an
inclusion and $\pr_2:E\oplus F\rightarrow F$
is a projection of Hermitian spaces.
\end{proposition}
\begin{proof}
Suppose that $E\oplus F$ is equipped with the
Hermitian metric $\|\cdot\|$ such that for
any $(x,y)\in E\oplus F$,
\[\|(x,y)\|^2=\|x-\varphi^*y\|_E^2+\|y\|_F^2\]
where $\|\cdot\|_E$ and $\|\cdot\|_F$ are
Hermitian metrics on $E$ and on $F$
respectively.  Clearly with this metric,
$\pr_2:E\oplus F\rightarrow F$ is a
projection of Hermitian spaces. Moreover,
$\pi^*$ is the identification of $F_0$ to
$(\Ker\pi)^{\perp}$, $i^*$ is the orthogonal
projection of $F$ onto $F_0$. Therefore,
$\varphi^*\varphi:E\rightarrow E$ is the
orthogonal projection of $E$ onto
$(\Ker\pi)^{\perp}$. Hence, for any vector
$x\in E$, we have
\[\|(x,\varphi(x))\|^2=\|x-\varphi^*\varphi(x)\|_E^2
+\|\varphi(x)\|_F^2=\|x\|_E^2.\]
\end{proof}

The assertion above works also in Hilbert
spaces case with the same choice of metric.
Furthermore, it can be generalized to the
family case. Suppose that $X$ is a space ringed in
$\mathbb R$-algebra (resp. smooth manifold)
and $E$, $F_0$ and $F$ are locally free
$\mathcal O_{X,\mathbb C}$-modules of finite
rank such that $F_0$ is a quotient of $E$ and
a submodule of $F$. We denote by $\varphi$ the canoical homomorphism defined by the composition of the projection from $E$ to $F_0$ and the inclusion of $F_0$ into $F$. If
$E$ and $F$ are equipped with continuous
(resp. smooth) Hermitian metrics such that
for any point $x\in X$, the quotient metric
on $F_{0,x}$ by the projection of $E_x$
coincides with the metric induced from that
of $F_x$, then there exists a continuous
(resp. smooth) Hermitian metric on $E\oplus
F$ such that for any point $x\in X$, the
graph of $\varphi_x$ defines an inclusion of
Hermitian spaces and the second projection
$E_x\oplus F_x\rightarrow F_x$ is a
projection of Hermitian spaces.

The arithmetic objects in
$(\mathbf{Vec}_{\mathbb C},\mathcal E,A)$ are
nothing other than Hermitian spaces. From
definition we see immediately that if a
linear mapping $\varphi:E\rightarrow F$ of
Hermitian spaces is compatible with
arithmetic structure, then the norm of
$\varphi$ must be smaller or equal to $1$.
The following proposition shows that the
converse is also true.

\begin{proposition}\label{Pro:decomposition de morphisme de petit norm}
Let $\varphi:E\rightarrow F$ be a linear map
of Hermitian spaces. If $\|\varphi\|\le 1$,
then there exists a Hermitian metric on
$E\oplus F$ such that in the decomposition
$\xymatrix{\relax
E\ar[r]^-{(\Id,\varphi)}&E\oplus
F\ar[r]^-{\pr_2}& F}$ of $\varphi$,
$(\Id,\varphi)$ is an inclusion of Hermitian
spaces and $\pr_2$ is a projection of
Hermitian spaces.
\end{proposition}
\begin{proof}
Since $\|\varphi\|\le 1$, we have
$\|\varphi^{*}\|\le 1$. Therefore, we obtain
the inequalities $\|\varphi^{*}\varphi\|\le
1$ and $\|\varphi\varphi^{*}\|\le 1$. Hence
$\Id_E-\varphi^{*}\varphi$ and
$\Id_F-\varphi\varphi^{*}$ are Hermitian
endomorphisms with positive eigenvalues. So
there exist two Hermitian endomorphisms with
positive eigenvalues $P$ and $Q$ of $E$ and
$F$ respectively such that
$P^2=\Id_E-\varphi^{*}\varphi$ and
$Q^2=\Id_F-\varphi\varphi^{*}$.

If $x$ is an eigenvector of
$\varphi\varphi^{*}$ associated to the
eigenvalue $\lambda$, then $\varphi^{*}x$ is
an eigenvector of $\varphi^{*}\varphi$
associated to the same eigenvalue. Therefore
$\varphi^{*}Qx=\sqrt{1-\lambda}\varphi^{*}x=P\varphi^{*}x$.
As $F$ is generated by eigenvectors of
$\varphi\varphi^*$, we have
$\varphi^{*}Q=P\varphi^*$. For the same
reason we have $Q\varphi=\varphi P$. Let
$R=\begin{pmatrix}
P&\varphi^*\\
\varphi&-Q
\end{pmatrix}$. As $R$ is clearly Hermitian, and
verifies \[R^2=\begin{pmatrix}
P^2+\varphi^*\varphi&P\varphi^{*}-\varphi^{*}Q\\
\varphi P-Q\varphi& \varphi\varphi^{*}+Q^2
\end{pmatrix}=\Id_{E\oplus F},\]
it is an isometry for the orthogonal
sum metric on $E\oplus F$. Let
$u:E\rightarrow E\oplus F$ be the mapping
which sends $x$ to
$\begin{pmatrix}x\\0\end{pmatrix}$. The
diagram
\[\xymatrix{\relax E\ar[r]^-{\varphi}\ar[d]_{u}& F\\
 E\oplus F\ar[r]_{R}& E\oplus
F\ar[u]_{\;\pr_2}}\] is commutative. The
endomorphism $\varphi^*\varphi$ is
auto-adjoint, there exists therefore an
orthonormal base $(x_i)_{1\le i\le n}$ of $E$
such that $\varphi^*\varphi
x_i=\lambda_ix_i$. Suppose that $0\le
\lambda_j<1$ for any $1\le j\le k$ and that
$\lambda_j=1$ for any $k<j\le n$. Let
$B:E\rightarrow E$ be the $\mathbb C$-linear
mapping such that
$B(x_j)=\sqrt{1-\lambda_j}x_j$ for $1\le j\le
k$ and that $B(x_j)=x_j$ for $j>k$.
Define $S=\begin{pmatrix} B&\varphi^*\\
0&\Id_F
\end{pmatrix}:E\oplus F\rightarrow E\oplus F$.
Since
$Ru=\begin{pmatrix}{P}\\{\varphi}\end{pmatrix}$
and since
\[(BP+\varphi^*\varphi)(x_i)=\sqrt{1-\lambda_i}
Bx_i+\lambda_i x_i=\begin{cases}
(1-\lambda_i)x_i+\lambda_ix_i=x_i,&1\le i\le
k,\\
0Bx_i+x_i=x_i,&k<i\le n,
\end{cases}\] the diagram
\[\xymatrix @R=1em @C=3pc {& E\oplus F\ar[dd]^S\ar[rd]^{\pr_2}\\
E\ar[ru]^{Ru}\ar[rd]_{\tau}&&F\\
&E\oplus F\ar[ru]_{\pr_2}}\] is commutative,
where $\tau=\begin{pmatrix}\Id_E\\
\varphi
\end{pmatrix}$. We equip $E\oplus F$
with the Hermitian product
$\left<\cdot,\cdot\right>_0$ such that, for
any $(\alpha,\beta)\in (E\oplus F)^2$, we
have
\[\left<\alpha,\beta\right>_0=\left<S^{-1}\alpha,
S^{-1}\beta\right>,\] where
$\left<\cdot,\cdot\right>$ is the orthogonal
direct sum of Hermitian products on $E$ and
on $F$. Then for any $(x,y)\in E\times E$,
\[\left<\tau(x),\tau(y)\right>_0=
\left<SRu(x),SRu(y)\right>_0=
\left<Ru(x),Ru(y)\right>=\left<u(x),u(y)\right>
=\left<x,y\right>.\] Finally, the kernel of
$\pr_2$ is stable by the action of $S$, so
the projections of
$\left<\cdot,\cdot\right>_0$ and of
$\left<\cdot,\cdot\right>$ by $\pr_2$ are the
same.
\end{proof}

From the proof of Proposition
\ref{Pro:decomposition de morphisme de petit
norm}, we see that a weaker form (the case
where $\|\varphi\|<1$) can be generalized to
the family case, no matter the family of
Hermitian metrics is continuous or smooth.

\subsection*{Ultranormed space}

\hskip\parindent Let $k$ be a field equipped
with a non-Archimedean absolute value
$\|\cdot\|$ under which $k$ is complete. We
denote by $\mathbf{Vec}_k$ the category of
finite dimensional vector spaces over $k$,
which is clearly an Abelian category. Let
$\mathcal E$ be the class of short exact
sequence in $\mathbf{Vec}_k$. For any finite
dimension vector space $E$ over $k$, we
denote by $A(E)$ the set of all ultranorms
(see \cite{Bourbaki81} for definition) on
$E$. Suppose that $h$ is an ultranorm on $E$.
If $i:E_0\rightarrow E$ is a subspace of $E$,
we denote by $i^*(h)$ the induced ultranorm
on $E_0$. If $\pi:E\rightarrow F$ is a
quotient space of $E$, we denote by
$\pi_*(h)$ the quotient ultranorm on $F$.
Then $(\mathbf{Vec}_k,\mathcal E,A)$ is an
arithmetic exact category. In particular, the
axiom ({\bf A}\ref{Axiome:A recollement}) is
justified by the following proposition, which
can be generalized without any difficulty to
Banach space case or family case.

\begin{proposition}
Let $\varphi:E\rightarrow F$ be a linear map
of vector spaces over $k$. Suppose that $E$
and $F$ are equipped respectively with the
ultranorms $h_E$ and $h_F$ such that
$\|\varphi\|\le 1$. If we equip $E\oplus F$
with the ultranorm $h$ such that, for any
$(x,y)\in E\oplus F$,
$h(x,y)=\max(h_E(x),h_F(y))$, then in the
decomposition $\xymatrix{\relax
E\ar[r]^-{(\Id,\varphi)}&E\oplus
F\ar[r]^-{\pr_2}&F}$ of $\varphi$, we have
$(\Id,\varphi)^*(h)=h_E$ and
$\pr_{2*}(h)=h_F$.
\end{proposition}
\begin{proof}
In fact, for any element $x\in E$,
$h(x,\varphi(x))=\max(h_E(x),h_F(\varphi(x)))=h_E(x)$
since $h_F(\varphi(x))\le\|\varphi\|h_E(x)\le
h_E(x)$. Furthermore, by definition it is
clear that $h_F=\pr_{2*}(h)$. Therefore the
proposition is true.
\end{proof}

\subsection*{Hermitian
vector bundles}

\hskip\parindent Let $K$ be a number field
and $\mathcal O_K$ be its integer ring. For
any scheme $\mathscr X$ of finite type and
flat over $\Spec\mathcal O_K$ such that
$\mathscr X_K$ is smooth, we denote by
$\mathbf{Vec}(\mathscr X)$ the category of
locally free modules of finite rank on
$\mathscr X$. If we denote by $\mathcal E$
the class of all short exact sequence of
coherent sheaves in $\mathbf{Vec}(\mathscr
X)$, then $(\mathbf{Vec}(\mathscr X),\mathcal
E )$ is an exact category. Let
$\Sigma_\infty$ be the set of all embeddings
of $K$ in $\mathbb C$. The space $\mathscr
X(\mathbb C)$ of complex points of $\mathscr
X$, which is a complex analytic manifold, can
be written as a disjoint union
\[\mathscr X(\mathbb C)=\coprod_{\sigma\in\Sigma_\infty}
\mathscr X_\sigma(\mathbb C),\] where
$\mathscr X_{\sigma}(\mathbb C)$ is the space
of complex points in $\mathscr
X\times_{\mathcal O_K,\sigma }\Spec\mathbb
C$. Notice that the complex conjugation of
$\mathbb C$ induces an involution
$F_\infty:\mathscr X(\mathbb
C)\rightarrow\mathscr X(\mathbb C)$ which
sends $\mathscr X_\sigma(\mathbb C)$ onto
$\mathscr X_{\overline\sigma}(\mathbb C)$.

We call {\it Hermitian vector bundle} on
$\mathscr X$ any pair $(E,h)$ where $E$ is an
object in $\mathbf{Vec}(\mathscr X)$ and
$h=(h_\sigma)_{\sigma\in\Sigma_\infty}$ is a
collection such that, for any
$\sigma\in\Sigma_\infty$, $h_\sigma$ is a
continuous Hermitian metric on
$E_\sigma(\mathbb C)$, $E_\sigma$ being
$E\otimes_{\mathcal O_K,\sigma}\mathbb C$,
subject to the condition that the collection
$h=(h_\sigma)_{\sigma\in\Sigma_\infty}$
should be invariant under the action of
$F_\infty$. The collection of Hermitian
metrics $h$ is called a {\it Hermitian
structure} on $E$. One can consult for
example \cite{Bost2001} and
\cite{Chambert} for details. If
$i:E_0\rightarrow E$ is an injective
homomorphism of $\mathcal O_{\mathscr
X}$-modules in $\mathbf{Vec}(\mathscr X)$, we
denote by $i^*(h)$ the collection of induced
metrics on $(E_{0,\sigma}(\mathbb
C))_{\sigma\in\Sigma_\infty}$; if
$\pi:E\rightarrow F$ is a surjective
homomorphism of $\mathcal O_{\mathscr
X}$-modules in $\mathbf{Vec}(\mathscr X)$, we
denote by $\pi_*(h)$ the collection of
quotient metric on $(F_\sigma(\mathbb
C))_{\sigma\in\Sigma_\infty}$. For any object
$E$ in $\mathbf{Vec}(\mathscr X)$, let $A(E)$
be the set of all Hermitian structures on
$E$. The family version of Proposition
\ref{Pro:inverse of projection - inclusion}
implies that $(\mathbf{Vec}(\mathscr
X),\mathcal E,A)$ is an arithmetic exact
category. The family version of Proposition
\ref{Pro:decomposition de morphisme de petit
norm} implies that, if $(E,h_E)$ and
$(F,h_F)$ are two Hermitian vector bundles
over $\mathscr X$ and if
$\varphi:E\rightarrow F$ is a homomorphism of
$\mathcal O_{\mathscr X }$-modules in
$\mathbf{Vec}(\mathscr X)$ such that, for any
$x\in \mathscr X(\mathbb C)$,
$\|\varphi_x\|<1$, then $\varphi$ is
compatible with arithmetic structures.

We say that a Hermitian structure
$h=(h_{\sigma})_{\sigma\in\Sigma_\infty}$ on
a vector bundle $E$ on $\mathscr X$ is {\it
smooth} if for any $\sigma\in\Sigma_\infty$,
$h_\sigma$ is a smooth Hermitian metric. For
any vector bundle $E$ on $\mathscr X$, let
$A_0(E)$ be the set of all smooth Hermitian
structures on $E$. Then
$(\mathbf{Vec}(\mathscr X),\mathcal E,A_0)$
is also an arithmetic exact category. If
$(E,h_E)$ and $(F,h_F)$ are two smooth
Hermitian vector bundles over $\mathscr X$,
then any homomorphism $\varphi:E\rightarrow
F$ which has norm $<1$ at every complex point
of $\mathscr X$ is compatible with arithmetic
structures.

\subsection*{Filtrations in an Abelian category}

\hskip\parindent Let $\mathcal C$ be an
essentially small Abelian category and
$\mathcal E$ be the class of short exact
sequences in $\mathcal C$. It is well known
that any finite projective limit (in
particular any fiber product) exists in
$\mathcal C$. Furthermore, any morphism in
$\mathcal C$ has an image, which is
isomorphic to the cokernel of its kernel, or
the kernel of its cokernel. For any object
$X$ in $\mathcal C$, we denote by $A(X)$ the
set\footnote{This is a set because $\mathcal
C$ is essentially small.} of isomorphism
classes of left continuous $I$-filtrations of
$X$, where $I$ is a totally ordered set, as
explained in the beginning of the second
section. For any left continuous
$I$-filtration $\mathcal F$ of $X$, we denote
by $[\mathcal F]$ the isomorphism class of
$\mathcal F$. If $u:X_0\rightarrow X$ is a
monomorphism, we define $u^*[\mathcal F]$ to
be the class of the inverse image
$u^*\mathcal F$. If $\pi:X\rightarrow Y$ is
an epimorphism, we define $\pi_*[\mathcal F]$
to be the class of the strong direct image
$\pi_*\mathcal F$.

We assert that $(\mathcal C,\mathcal E,A)$ is
an arithmetic exact category. In fact, the
axioms ({\bf A}\ref{Axiome:A A(0) est
singleton})
--- ({\bf A}\ref{Axiome:A iosmorphisme}) are
clearly satisfied.  We now verify the axiom
({\bf A}\ref{Axiome:A carre cartesien}).
Consider the diagram \eqref{Equ:A carre
cartesien} in Definition \ref{Def:categorie exacte arithematque}, which is the right sagittal
square of the following diagram
\eqref{Equ:verification of Axiom 4}. Suppose
given an $I$-filtration $\mathcal F$ of $Y$.
For any $i\in I$, we note $Y_i=\mathcal F(i)$
and we denote by $b_i:Y_i\rightarrow Y$ the
canonical monomorphism.
\begin{equation}\label{Equ:verification of Axiom 4}
\xymatrix{\relax &Z_i\ar@{
>->}[rr]^-{c_i}
\ar@{ >->}'[d]^-{\;v_i}[dd]&& Z\ar@{ >->}[dd]^-{v}\\
X_i\ar@{.>}[ru]^-{p_i}\ar@{
>->}[dd]_-{u_i} \ar@{
>->}[rr]_<<<<<<<<<{a_i}&& X\ar@{->>}[ru]_-{p}
\ar@{ >->}[dd]^<<<<<<<<<{u}\\
&W_i\ar@{ >->}'[r]^-{d_i}[rr]&&W\\
Y_i\ar@{->>}[ru]^-{q_i}\ar@{
>->}[rr]_-{b_i}&&Y \ar@{->>}[ru]_-{q}}\end{equation}
Let $d_i:W_i\rightarrow W$ be the image of
$qb_i$ in $W$ and $q_i:Y_i\rightarrow W_i$ be
the canonical epimorphism. Let
$(Z_i,c_i,v_i)$ be the fiber product of $v$
and $d_i$, and $(X_i,a_i,u_i)$ be the fiber
product of $u$ and $b_i$. Therefore, in the
diagram \eqref{Equ:verification of Axiom 4},
the two coronal square and the right sagittal
square are cartesian, the inferior square is
commutative. As
$vpa_i=qua_i=qb_iu_i=d_iq_iu_i$, there exists
a unique morphism $p_i:X_i\rightarrow Z_i$
such that $c_ip_i=pa_i$ and that
$v_ip_i=q_iu_i$. It is then not hard to verify
that the left sagittal square is cartesian,
therefore $p_i$ is an epimorphism, so $Z_i$
is the image of $pa_i$. The axiom ({\bf
A}\ref{Axiome:A carre cartesien}) is
therefore verified. Finally, the verification
of the axiom ({\bf A}\ref{Axiome:A
recollement}) follows from the following
proposition.

\begin{proposition}\label{Pro:filtrations is arithmetic
exact category} Let $X$ and $Y$ be two
objects in $\mathcal C$ and let $\mathcal F$
(resp. $\mathcal G$) be an $I$-filtration of
$X$ (resp. $Y$). If $f:X\rightarrow Y$ is a
morphism which is compatible with the
filtrations $(\mathcal F,\mathcal G)$, then
there exists a filtration $\mathcal H$ on
$X\oplus Y$ such that $\Gamma_f^*\mathcal
H=\mathcal F$ and $\pr_{2*}\mathcal
H=\mathcal G$, where
$\Gamma_f=(\Id,f):X\rightarrow X\oplus Y$ is
the graph of $f$ and $\pr_2:X\oplus
Y\rightarrow Y$ is the projection onto the
second factor.
\end{proposition}
\begin{proof}
Let $\mathcal H$ be the filtration such that
$\mathcal H(i)=\mathcal F(i)\oplus\mathcal
G(i)$. Clearly it is left continuous, and
$\pr_{2\flat}\mathcal H=\mathcal G$.
Therefore $\pr_{2*}\mathcal H=\mathcal
G^{l}=\mathcal G$. Moreover, for any $i\in
I$, consider the square
\begin{equation}\label{Equ:square fiber product}
\xymatrix{\relax \mathcal
F(i)\ar[r]^-{\phi_i}\ar[d]_-{(\Id,f_i)} &
X\ar[d]^-{\;(\Id,f)}\\ \mathcal
F(i)\oplus\mathcal
G(i)\ar[r]_-{\Phi_i}&X\oplus Y
}\end{equation} where $\phi_i:\mathcal
F(i)\rightarrow X$ and
$\Phi_i=\phi_i\oplus\psi_i:\mathcal
F(i)\oplus\mathcal G(i)\rightarrow X\oplus Y$
are canonical inclusions, $f_i:\mathcal
F(i)\rightarrow\mathcal G(i)$ is the morphism
through which the restriction of
$f$ on $\mathcal F(i)$ (i.e., $f\phi_i$)  factorizes.
Then the square \eqref{Equ:square fiber
product} is commutative. Suppose that
$\alpha:Z\rightarrow X$ and
$\beta=(\beta_1,\beta_2):Z\rightarrow\mathcal
F(i)\oplus\mathcal G(i) $ are two morphisms
such that $(\Id,f)\alpha=\Phi_i\beta$.
\begin{equation}\label{Equ:produit fibre fiplusgi}
\xymatrix{\relax
Z\ar@{.>}[rd]|-{\beta_1}\ar@/^1pc/[rrd]^-{\alpha}
\ar@/_1pc/[rdd]_-{\beta}\\
&\mathcal
F(i)\ar[r]^-{\phi_i}\ar[d]^-{\;(\Id,f_i)} &
X\ar[d]^-{\;(\Id,f)}\\ &\mathcal
F(i)\oplus\mathcal
G(i)\ar[r]_-{\Phi_i}&X\oplus Y}\end{equation}
Then we have $\alpha=\phi_i\beta_1$ and
$f\alpha=\psi_i\beta_2$. So
\[\psi_i\beta_2=f\alpha=f\phi_i\beta_1=\psi_if_i\beta_1.\]
As $\psi_i$ is a monomorphism, we obtain that
$f_i\beta_1=\beta_2$. So
$\beta_1:Z\rightarrow\mathcal F(i)$ is the
only morphism such that the diagram
\eqref{Equ:produit fibre fiplusgi} commutes.
Hence we get $\mathcal F=(\Id,f)^*\mathcal
H$.
\end{proof}

Notice that the category of arithmetic
objects $\mathcal C_A$ is equivalent to the
category $\mathbf{Fil}^{I,l}(\mathcal C)$ of
left continuous filtrations. Moreover, there
exist some variants of $(\mathcal C,\mathcal
E,A)$. For example, if for any object $X$ in
$\mathcal C$, we denote by $A_0(X)$ the set
of isomorphism classes of $I$-filtrations
which are separated, exhaustive, left
continuous and of finite length. Then
$(\mathcal C,\mathcal E,A_0)$ is also an
arithmetic exact category. Furthermore, the
category $\mathcal C_{A_0}$ is equivalent to
$\mathbf{Fil}^{I,\mathrm{self}}(\mathcal C)$,
the full subcategory of
$\mathbf{Fil}^{I,l}(\mathcal C)$ consisting
of filtrations which are {\bf s}eparated,
{\bf e}xhaustive, {\bf l}eft continous and of
{\bf f}inite length.

\section{Harder-Narasimhan categories}

\hskip\parindent In this section we introduce
the formalism of Harder-Narasimhan
filtrations (indexed by $\mathbb R$) on
arithmetic exact categories. Let $(\mathcal
C,\mathcal E,A)$ be an arithmetic exact
category. We say that an arithmetic object
$(X,h)$ is {\it non-zero} if $X$ is non-zero
in $\mathcal C$. Since $\mathcal C$ is
essentially small, the isomorphism classes of
objects in $\mathcal C_{A}$ form a set.

We denote by $\mathcal E_A$ the class of
diagrams of the form
\[\xymatrix{\relax 0\ar[r]&(X',h')\ar[r]^-i&(X,h)\ar[r]^-p&(X'',h'')
\ar[r]&0}\] where $(X',h')$, $(X,h)$ and
$(X'',h'')$ are arithmetic objects and
\[\xymatrix{\relax 0\ar[r]&X'\ar[r]^-i&X\ar[r]^-p&X''\ar[r]&0}\]
is a diagram in $\mathcal E$ such that
$h'=i^*(h)$ and $h''=p_*(h)$.

Let $K_0(\mathcal C,\mathcal E,A)$ be the
free Abelian group generated by isomorphism
classes in $\mathcal C_A$, modulo the
subgroup generated by elements of the form
$[(X,h)]-[(X',h')]-[(X'',h'')]$, where
\[\xymatrix{\relax 0\ar[r]&(X',h')\ar[r]^-i
&(X,h)\ar[r]^-p&(X'',h'')\ar[r]&0}\] is a
diagram in $\mathcal E_A$, in other words,
$\xymatrix{\relax
0\ar[r]&X'\ar[r]^-i&X\ar[r]^-p&X''\ar[r]&0}$,
and $i^*(h)=h'$, $p_*(h)=h''$. The group
$K_0(\mathcal C,\mathcal E,A)$ is called the
{\it Grothendieck group} of the arithmetic
exact category $(\mathcal C,\mathcal E,A)$.
We have a ``{\it forgetful}'' homomorphism
from $K_0(\mathcal C,\mathcal E,A)$ to
$K_0(\mathcal C,\mathcal E)$, the
Grothendieck group\footnote{Which is, by
definition, the free Abelian group generated
by isomorphism classes in $\mathcal C$,
modulo the sub-group generated by elements of
the form $[X]-[X']-[X'']$, where
$\xymatrix{0\ar[r]&X'\ar[r]&X\ar[r]&X''\ar[r]&0}$
is a diagram in $\mathcal E$.} of the exact
category $(\mathcal C,\mathcal E)$, which
sends $[(X,h)]$ to $[X]$.

In order to establish the semi-stability of
arithmetic objects and furthermore the
Harder-Narasimhan formalism, we need two
auxiliary homomorphisms of groups. The first
one, from $K_0(\mathcal C,\mathcal E,A)$ to
$\mathbb R$, is called a {\it degree
function} on $(\mathcal C,\mathcal E,A)$; and
the second one, from $K_0(\mathcal C,\mathcal
E)$ to $\mathbb Z$, which takes strictly
positive values on elements of the form $[X]$
with $X$ non-zero, is called a {\it rank
function} on $(\mathcal C,\mathcal E)$.

Now let $\widehat\deg:K_0(\mathcal C,\mathcal
E,A )\rightarrow\mathbb R$ be a degree
function on $(\mathcal C,\mathcal E,A)$ and
$\rang:K_0(\mathcal C,\mathcal
E)\rightarrow\mathbb Z$ be a rank function on
$(\mathcal C,\mathcal E)$. For any arithmetic
object $(X,h)$ in $(\mathcal C,\mathcal E,
A)$, we shall use the expressions
$\widehat\deg(X,h)$ and $\rang(X)$ to denote
$\widehat\deg([(X,h)])$ and $\rang([X])$, and
call them the {\it arithmetic degree} and
the {\it rank} of $(X,h)$ respectively. If
$(X,h)$ is non-zero, the quotient
$\widehat\mu(X,h)=\widehat\deg(X,h)/\rang(X)$
is called the {\it arithmetic slope} of
$(X,h)$. We say that a non-zero arithmetic
object $(X,h)$ is {\it semistable} if for
any non-zero arithmetic subobject $(X',h')$
of $(X,h)$, we have
$\widehat\mu(X',h')\le\widehat\mu(X,h)$.

The following proposition provides some basic
properties of arithmetic degrees and of
arithmetic slopes.

\begin{proposition}\label{Pro:propriete preliminaire de degre}
Let us keep the notations above.
\begin{enumerate}[1)]
\item If $\xymatrix{0\ar[r]&(X',h')\ar[r]&(X,h)\ar[r]&(X'',h'')\ar[r]&0}$
is a diagram in $\mathcal E_A$, then
\[\widehat\deg(X,h)=\widehat\deg(X',h')+\widehat\deg(X'',h'').\]
\item If $(X,h)$ is an arithmetic object of
rank $1$, then it is semistable.
\item Any non-zero arithmetic
object $(X,h)$ is semistable if and only if
for any non-trivial arithmetic quotient
$(X'',h'')$ (i.e., $X''$ does not reduce to
zero and is not canonically isomorphic to $X$),
we have
$\widehat\mu(X,h)\le\widehat\mu(X'',h'')$.
\end{enumerate}
\end{proposition}
\begin{proof}
Since $\widehat\deg$ is a homomorphism from
$K_0(\mathcal C,\mathcal E,A)$ to $\mathbb
R$, 1) is clear.

2) If $(X',h')$ is an arithmetic subobject
of $(X,h)$, then it fits into a diagram
\[\xymatrix{\relax 0\ar[r]&(X',h')\ar[r]^-f&(X,h)\ar[r]&(X'',h'')\ar[r]&0}\]
 in
$\mathcal C_A$.
Since $X'$ is non-zero, $\rang(X')\ge 1$.
Therefore $\rang(X'')=0$ and hence $X''=0$.
In other words, $f$ is an isomorphism. So we
have $\widehat\mu(X',h')=\widehat\mu(X,h)$.

3) For any diagram
\[\xymatrix{0\ar[r]&(X',h')\ar[r]&(X,h)\ar[r]&(X'',h'')\ar[r]&0}\]
in $\mathcal E_A$, $(X'',h'')$ is non-trivial
if and only if $(X',h')$ is non-trivial. If
$(X',h')$ and $(X'',h'')$ are both
non-trivial, we have the following equality
\[\widehat{\mu}(X,h)=\frac{\rang(X')}{\rang(X)}\widehat{\mu}(X',h')
+\frac{\rang(X'')}{\rang(X)}\widehat\mu(X'',h'').\]
Therefore
$\widehat{\mu}(X',h')\le\widehat\mu(X,h)\Longleftrightarrow
\widehat{\mu}(X'',h'')\ge\widehat\mu(X,h)$.
\end{proof}

We are now able to introduce
conditions ensuring the existence and the
uniqueness of Harder-Narasimhan ``flag''. The
conditions will be proposed as axioms in the
coming definition, and in the theorem which
follows, we shall prove the existence and the
uniqueness of Harder-Narasimhan ``flag''.

\begin{definition}\label{Def:categorie de HN arithmetique}
Let $(\mathcal C,\mathcal E,A)$ ba an
arithmetic exact category,
$\widehat\deg:K_0(\mathcal C,\mathcal
E,A)\rightarrow\mathbb R$ be a degree function and
$\rang:K_0(\mathcal C,\mathcal
E)\rightarrow\mathbb Z$ be a rank function. We
say that $(\mathcal C,\mathcal
E,A,\widehat\deg,\rang)$ is a {\it
Harder-Narasimhan category} if the following
two axioms are verified:
\begin{enumerate}[({\bf HN}1)]
\item\label{Axiom:HNA sous fibre destabilisant}
For any non-zero arithmetic object $(X,h)$,
there exists an arithmetic subobject
$(X_{\des},h_{\des})$ of $(X,h)$ such that
\[\hskip-1 cm\widehat\mu(X_{\des},h_{\des})=\sup\{\widehat\mu(X',h')\;|\;
(X',h')\text{ is a non-zero arithmetic subobject of
}(X,h)\}.\] Furthermore, for any non-zero arithmetic
subobject $(X_0,h_0)$ of $(X,h)$ such that
$\widehat\mu(X_0,h_0)=\widehat\mu(X_{\des},h_{\des})$,
there exists an admissible monomorphism
$f:X_0\rightarrow X_{\des}$ such that the
diagram
\[\xymatrix{\relax X_0\ar[r]^-{f}\ar[rd]_-j&X_{\des}\ar[d]^-i\\&
X}\] is commutative and that
$f^*(h_{\des})=h_0$, where $i$ and $j$ are
canonical admissible monomorphisms.
\item\label{Axiom:HNA comparaison de pente pour des objet semistable}
If $(X_1,h_1)$ and $(X_2,h_2)$ are two
semistable arithmetic objects such that
$\widehat\mu(X_1,h_1)>\widehat\mu(X_2,h_2)$,
there exists {\bf no} non-zero morphism from $X_1$
to $X_2$ which is compatible with arithmetic
structures.
\end{enumerate}
\end{definition}

With the notations of Definition
\ref{Def:categorie de HN arithmetique}, if
$(X,h)$ is a non-zero arithmetic object, then
$(X_{\des},h_{\des})$ is a semistable
arithmetic object. If in addition $(X,h)$ is
not semistable, we say that
$(X_{\des},h_{\des})$ is the arithmetic
subobject which {\it destabilizes} $(X,h)$.

\begin{theorem}\label{Thm:suite de Harder-Narasimhan arithmetique}
Let $(\mathcal C,\mathcal
E,A,\widehat{\deg},\rang)$ be a
Harder-Narasimhan category. If $(X,h)$ is a
non-zero arithmetic object, then there exists
a sequence of admissible monomorphisms in
$\mathcal C$:
\begin{equation}\label{Equ:suite de
HNA}\xymatrix{0=X_0\ar[r]&X_1\ar[r]&\cdots\ar[r]&X_{n-1}\ar[r]&X_n=X},\end{equation}
unique up to a unique isomorphism, such that,
if for any integer $0\le i\le n$, we denote
by $h_i$ the induced arithmetic structure
(from $h$) on $X_i$ and if we equip, for any
integer $1\le j\le n$, $X_{j}/X_{j-1}$ with
the quotient arithmetic structure (of
$h_{j}$), then
\begin{enumerate}[1)]
\item for any integer $1\le j\le n$, the arithmetic
object $\overline{X_{j}/X_{j-1}}$ defined
above is semistable;
\item we have the inequalities $\widehat\mu(\overline{X_1/X_0})>
\widehat\mu(\overline{X_2/X_1})
>\cdots>\widehat\mu(\overline{X_n/X_{n-1}})$.
\end{enumerate}
\end{theorem}
\begin{proof}
First we prove the existence by induction on
the rank $r$ of $X$. The case where $(X,h)$
is semistable is trivial, and {\it a
fortiori} the existence is true for $r=1$.
Now we consider the case where $(X,h)$ isn't
semistable. Let
$(X_1,h_1)=(X_{\des},h_{\des})$. It's a
semistable arithmetic object, and $X'=X/X_1$
is non-zero. The rank of $X'$ being strictly
smaller than $r$, we can therefore apply the
induction hypothesis on $(X',h')$, where $h'$
is the quotient arithmetic structure. We then
obtain a sequence of admissible monomorphisms
\[\xymatrix{\relax 0=X_1'\ar[r]^-{f_1'}&
X_2'\ar[r]&\cdots\ar[r]&
X_{n-1}'\ar[r]^-{f_{n-1}'}&X_n'=X'}\]
verifying the desired condition.

Since the canonical morphism from $X$ to $X'$
is an admissible epimorphism, for any $1\le
i\le n$, if we note $X_i=X\times_{X'}X_i'$,
then by the axiom ($\mathbf{Ex}$\ref{Axiom:
ex epi stable par probduit}), the projection
$\pi_i:X_i\rightarrow X_i'$ is an admissible
epimorphism. For any integer $1\le i<n$, we
have a canonical morphism from $X_i$ to
$X_{i+1}$ and the square
\begin{equation}\label{Equ:carre cartesienne exacge}\xymatrix{\relax X_i\ar[r]^{f_i}\ar@{->>}[d]_-{\pi_i}&X_{i+1}\ar@{->>}[d]^-{\pi_{i+1}}\\
X_i'\ar[r]_{f_i'}&X_{i+1}'}\end{equation} is
cartesian. Since $f_i'$ is an monomorphism,
also is $f_i$ (cf. \cite{Maclane71} V. 7). On
the other hand, since the square
\eqref{Equ:carre cartesienne exacge} is
cartesian, $f_i$ is the kernel of the
composed morphism
\[\xymatrix{\relax X_{i+1}\ar@{->>}[r]^-{\pi_{i+1}}&X_{i+1}'\ar@{->>}[r]^-{p_i}&X_{i+1}'/X_i'},\]
where $p_i$ is the canonical morphism. Since
$\pi_{i+1}$ and $p_i$ are admissible
epimorphisms, also is $p_i\pi_{i+1}$ (see
axiom ($\mathbf{Ex}$\ref{Axiom:ex composition
invariant epi et mono})). Therefore $f_i$ is
an admissible monomorphism. Hence we obtain a
commutative diagram
\[\xymatrix{\relax 0=X_0\ar[r]&X_1\ar[r]^-{f_1}\ar[d]_-{\pi_1}&X_2
\ar[r]\ar[d]_{\pi_2}&\cdots\ar[r]&X_{n-1}\ar[r]^-{f_{n-1}}\ar[d]^{\pi_{n-1}}
&**[r]X_n=X\ar[d]^{\pi}\\
&**[l]0=X_1'\ar[r]_-{f_1'}&X_2'\ar[r]&\cdots\ar[r]&X_{n-1}'\ar[r]_-{f_{n-1}'}&**[r]X_n'=X'}\]
where the horizontal morphisms in the lines
are admissible monomorphisms and the vertical
morphisms are admissible epimorphisms.
Furthermore, for any integer $1\le i\le n-1$,
we have a natural isomorphism $\varphi_i$
from $X_{i+1}/X_i$ to $X_{i+1}'/X_i'$. We
denote by $g_i$ (resp. $g_i'$) the canonical
morphism from $X_i$ (resp. $X_i'$) to $X$
(resp. $X'$). Let $h_i=g_i^*(h)$ (resp.
$h_i'={g_i'}^*(h')$) be the induced
arithmetic structure on $X_i$ (resp. $X_i'$).
After the axiom ({\bf A}\ref{Axiome:A carre
cartesien}),
${\pi_i}_*(h_i)={\pi_{i}}_*f_i^*(h)=f_i'^*\pi_*(h)=h_i'$.
Therefore ${\varphi_i}_*$ sends the quotient
arithmetic structure on $X_{i+1}/X_i$ to that
on $X_{i+1}'/X_i'$. Hence the arithmetic
object $\overline{X_{i+1}/X_i}$ is
semistable and we have the equality
$\widehat\mu(\overline{X_{i+1}/X_i})=\widehat\mu(\overline{X_{i+1}'/X_i'})
$. Finally, since $\overline X_1=\overline
X_{\des}$, we have
\[\widehat\mu(\overline{X_2/X_1})=\frac{\rang(X_2)
\widehat\mu(\overline
X_2)-\rang(X_1)\widehat\mu(\overline
X_1)}{\rang(X_2)-\rang(X_1)}<\widehat{\mu}(\overline
X_1).\] Therefore the sequence
$\xymatrix{0=X_0\ar[r]&X_1\ar[r]&\cdots\ar[r]&X_{n-1}\ar[r]&X_n=X}$
satisfies the desired conditions.

We then prove the uniqueness of the sequence
(\ref{Equ:suite de HNA}). By induction we
only need to prove that $\overline
X_1\cong\overline X_{\des}$. Let $i$ be the
first index such that the canonical morphism
$X_{\des}\rightarrow X$ factorizes through
$X_{i+1}$. The composed morphism
$X_{\des}{\rightarrow}X_{i+1}{\rightarrow}
X_{i+1}/X_i$ is then non-zero. Since
$\overline X_{\des}$ and
$\overline{X_{i+1}/X_i}$ are semistable, we
have $\widehat\mu(\overline
X_{\des})\le\widehat\mu(\overline{X_{i+1}/X_i})$.
This implies $i=0$ and $\widehat\mu(\overline
X_{\des})=\widehat\mu(\overline X_1)$.
Therefore the morphism $X_1\rightarrow X$
factorizes through $X_{\des}$. So we have
$X_{\des}\cong X_1$.
\end{proof}

From the proof above we see that the axiom
({\bf HN}\ref{Axiom:HNA sous fibre
destabilisant}) suffices for the existence.
It is the axiom ({\bf HN}\ref{Axiom:HNA
comparaison de pente pour des objet
semistable}) which ensures the uniqueness.

\begin{definition}
With the notations of Theorem \ref{Thm:suite
de Harder-Narasimhan arithmetique}, the
sequence (\ref{Equ:suite de HNA}) is called
the {\it Harder-Narasimhan sequence} of the
(non-zero) arithmetic object $(X,h)$.
Sometimes we write instead
\[\xymatrix{0=\overline X_0\ar[r]&\overline X_1\ar[r]&\cdots
\ar[r]&\overline X_{n-1}\ar[r]&\overline
X_n=\overline X}\] for underlining
the arithmetic structures. The real numbers
$\widehat\mu(\overline X_1)$ and
$\widehat\mu(\overline{X/X_{n-1}})$ are
called respectively the {\it maximal slope}
and the {\it minimal slope} of $\overline X
$, denoted by $\widehat\mu_{\max}(\overline
X)$ and $\widehat\mu_{\min}(\overline X)$. We point out that for any
integer $1\le i\le n$,
\[\xymatrix{0=X_0\ar[r]&X_1\ar[r]&\cdots\ar[r]
&X_{i-1}\ar[r]&X_i}\] is the
Harder-Narasimhan sequence of $\overline
X_i$. Therefore we have
$\widehat{\mu}_{\min}(\overline
X_i)=\widehat{\mu}(\overline{X_i/X_{i-1}})$.
Finally, we define by convention
$\widehat{\mu}_{\max}(0)=-\infty$ and
$\widehat{\mu}_{\min}(0)=+\infty$.
\end{definition}

\begin{corollary}
Let $(\mathcal C,\mathcal
E,A,\widehat{\deg},\rang)$ be a
Harder-Narasimhan category and $\overline X$
be a non-zero arithmetic object.
\begin{enumerate}[1)]
\item For any non-zero arithmetic subobject $\overline
Y$ of $\overline X$, we have
$\widehat{\mu}_{\max}(\overline
Y)\le\widehat{\mu}_{\max}(\overline X)$.
\item For any non-zero arithmetic quotient $\overline Z$
of $\overline X$, we have
$\widehat{\mu}_{\min}(\overline
Z)\ge\widehat{\mu}_{\min}(\overline X)$.
\item We have the inequalities $\widehat{\mu}_{\min}(\overline X)
\le\widehat{\mu}(\overline
X)\le\widehat{\mu}_{\max}(\overline X)$.
\end{enumerate}
\end{corollary}
\begin{proof}
Let $0=X_0\longrightarrow
X_1\longrightarrow\cdots\longrightarrow
X_{n-1}\longrightarrow X_n=X$ be the
Harder-Narasimhan sequence of $\overline X$.

1) After replacing $\overline Y$ by
$\overline Y_{\des}$ we may suppose that
$\overline Y$ is semistable. Let $i$ be the
first index such that the canonical morphism
$Y\rightarrow X$ factorizes through
$X_{i+1}$. The composed morphism
$Y\rightarrow X_{i+1}\rightarrow X_{i+1}/X_i$
is non-zero and compatible with arithmetic
structures. Therefore
\[\widehat{\mu}(Y)\le\widehat{\mu}(X_{i+1}/X_i)
\le\widehat{\mu}_{\max}(X).\]

2) After replacing $\overline Z$ by a
semistable quotient we may suppose that
$\overline Z$ is itself semistable. Let
$f:X\rightarrow Z$ be the canonical morphism.
It is an admissible epimorphism. Let $i$ be
the smallest index such that the composed
morphism $X_{i+1}\rightarrow
X\stackrel{f}{\rightarrow}Z $ is non-zero.
Since the composed morphism $X_i\rightarrow
X\stackrel{f}{\rightarrow}Z$ is zero, we
obtain a non-zero morphism from $X_{i+1}/X_i$
to $Z$ which is compatible with arithmetic
structures after Axiom ({\bf A}\ref{Axiome:A
carre cartesien}).
\[\xymatrix{X_{i+1}\ar@{ >->}[r]\ar@{->>}[d]
&X\ar@{->>}[d]\\ X_{i+1}/X_i\ar@{
>->}[r]&X/X_i\ar@{->>}[r]&Z}\]
Therefore $\widehat{\mu}(\overline
Z)\ge\widehat{\mu}(\overline{X_{i+1}/X_i})
\ge\widehat{\mu}_{\min}(\overline X)$.

3) We have $\widehat{\deg}(\overline
X)=\displaystyle\sum_{i=1}^{n}\widehat{\deg}(\overline{X_i/X_{i-1}})$.
Therefore
\[\widehat{\mu}(X)=\sum_{i=1}^n\frac{\rang(
X_i/X_{i-1})}{\rang(X)}\widehat{\mu}(\overline{X_i/X_{i-1}})
\in\big[\widehat{\mu}_{\min}(\overline
X),\widehat{\mu}_{\max}(\overline X)\big].\]
\end{proof}

It is well known that if $E$ and $F$ are two
vector bundles on a smooth projective curve
$C$ such that $\mu_{\min}(E)>\mu_{\max}(F)$,
then there isn't any non-zero homomorphism
from $E$ to $F$. The following result
(Proposition \ref{Pro:mu max X ge mu min Y si
X vers Y non nul arith}) generalizes this
fact to Harder-Narasimhan categories.

\begin{lemma}
\label{Lem:compose un epi est compatible
implie compatible} Let $(\mathcal C,\mathcal
E,A)$ be an arithmetic exact category. Suppose
that any epimorphism in $\mathcal C$ has a
kernel. Let $(X,h_X)$ and $(Z,h_Z)$ be two
arithmetic objects, $(Y,h_Y)$ be an
arithmetic quotient of $(X,h_X)$, and
$f:Y\rightarrow Z$ be a morphism in $\mathcal
C$. Denote by $\pi:X\rightarrow Y$ the
canonical admissible epimorphism. The
morphism $f$ is compatible with arithmetic
structures if and only if it is the case for
$f\pi$.
\end{lemma}
\begin{proof}
Since $\pi$ is compatible with arithmetic
structures, the compatibility of $f$ with
arithmetic structures implies that of $f\pi$.
It then suffices to verify the converse
assertion. By definition there exists an
arithmetic object $(W,h_W)$ and a
decomposition $\xymatrix{X\ar@{
>->}[r]^i&W\ar@{->>}[r]^p &Z}$ of $f\pi$
such that $i^*h_W=h_X$ and $p_*h_W=h_Z$. Let
$T$ be the fiber coproduct of $i$ and $\pi$
and let $j:Y\rightarrow T$ and
$q:W\rightarrow T$ be canonical morphisms.
After Axiom ({\bf Ex}\ref{Axiom:ex mono
stable par coproduit}), $j$ is an admissible
monomorphism. Let $\tau:U\rightarrowtail X$
be the kernel of $\pi$. We assert that
$q=\Coker(i\tau)$. On one hand, we have
$qi\tau=j\pi\tau=0$. On the other hand, if
$\alpha:W\rightarrow V$ is a morphism in
$\mathcal C$ such that $\alpha i\tau=0$, then
there exists a unique morphism
$\beta:Y\rightarrow V$ such that
$\beta\pi=\alpha i$ since $\pi$ is a cokernel
of  $\tau$. Therefore, there exits a unique
morphism $\gamma:T\rightarrow V$ such that
$\gamma q=\alpha$. So $q$ is a cokernel of
$i\tau$, hence an admissible epimorphism. The
morphisms $p:W\twoheadrightarrow Z$ and
$f:Y\rightarrow Z$ induce a morphism
$g:T\rightarrow Z$:
\[\xymatrix{U\ar@{ >->}[d]_\tau\\X\ar@{ >->}[r]^i\ar@{->>}[d]_{\pi}&W\ar@{->>}[d]_q\ar@{->>}[rd]^p\\
Y\ar@{ >->}[r]_j&T\ar[r]_g&Z}\] Since $g$ is
an epimorphism, by hypothesis it has a
kernel. After Axiom ({\bf Ex}\ref{Axiom: ex
composition ep mon implie un epi mono}), it
is an admissible epimorphism. Finally if we
denote by $h_T$ the arithmetic structure
$q_*h_W$ on $T$, we have
$g_*(h_T)=p_*(h_W)=h_Z$ and
$j^*(h_T)=\pi_*(i^*h_W)=\pi_*(h_X)=h_Y$.
\end{proof}

\begin{proposition}
\label{Pro:mu max X ge mu min Y si X vers Y
non nul arith} Let $(\mathcal C,\mathcal
E,A,\widehat{\deg},\rang)$ be a
Harder-Narasimhan category. Suppose that
any epimorphism in $\mathcal C$ has a kernel.
If $\overline X$ and $\overline Y$ are two
arithmetic objects and if $f:\overline
X\rightarrow\overline Y$ is a non-zero
morphism compatible with arithmetic
structures, then
$\widehat{\mu}_{\min}(\overline
X)\le\widehat{\mu}_{\max}(\overline Y)$.
\end{proposition}
\begin{proof}
Let
$\xymatrix{0=\overline X_0\ar[r]&\overline X_1\ar[r]&\cdots
\ar[r]&\overline X_{n-1}\ar[r]&\overline
X_n=\overline X}$ be the Harder-Narasimhan
sequence of $\overline X$. For any integer
$0\le i\le n$, let $h_i:\overline
X_i\rightarrow \overline X$ be the canonical
monomorphism. Let $1\le j\le n$ be the first
index such that $fh_j$ is non-zero. Since
$fh_{j-1}=0$, the morphism $fh_j$
factorizes through $X_j/ X_{j-1}$, so we get
a non-zero morphism $g$ from $X_j/X_{j-1}$ to
$Y$. After Lemma \ref{Lem:compose un epi est
compatible implie compatible}, $g$ is
compatible with arithmetic structures. Let
\[\xymatrix{0=\overline Y_0\ar[r]&\overline Y_1\ar[r]&\cdots\ar[r]&\overline Y_{m-1}\ar[r]&\overline Y_m}\]
be the Harder-Narasimhan sequence of
$\overline Y$. Let $1\le k\le n$ be the first
index such that $g$ factorizes through $Y_k$.
If $\pi:Y_k\rightarrow Y_k/Y_{k-1}$ is the
canonical morphism, then $\pi g$ is non-zero
since $g$ doesn't factorize through
$Y_{k-1}$. Furthermore, it is compatible with
arithmetic structures. Therefore, we have
\[\widehat{\mu}_{\min}(\overline X)\le
\widehat{\mu}(\overline{ X_j/
X_{j-1}})\le\widehat{\mu}(\overline{ Y_k/
Y_{k-1}})\le \widehat{\mu}_{\max}(\overline
Y).\]
\end{proof}

\begin{corollary}\label{Cor:inegality de pentes}
Keep the notations and the hypothesis of Proposition \ref{Pro:mu
max X ge mu min Y si X vers Y non nul arith}.
\begin{enumerate}[1)]
\item If in addition $f$ is monomorphic, then
$\widehat{\mu}_{\max}(\overline
X)\le\widehat{\mu}_{\max}(\overline Y)$.
\item If in addition $f$ is epimorphic, then $\widehat{\mu}_{\min}(\overline
X)\le\widehat{\mu}_{\min}(\overline Y)$.
\end{enumerate}
\end{corollary}
\begin{proof}
Suppose that $f$ is monomorphic. Let
$i:X_{\des}\rightarrow X$ be the canonical
morphism. Then the composed morphism
$fi:\overline X_{\des}\rightarrow\overline Y$
is non-zero and compatible with arithmetic
structures. Therefore
$\widehat{\mu}_{\max}(\overline
X)=\widehat{\mu}_{\min}(\overline X_{\des})
\le\widehat{\mu}_{\max}(\overline Y)$. The
proof of the other assertion is similar.
\end{proof}

If the arithmetic structure $A$ is trivial,
then any morphism in $\mathcal C$ is
compatible with arithmetic structures.
Therefore in this case we may remove the
hypothesis on the existence of kernels in
Proposition \ref{Pro:mu max X ge mu min Y si
X vers Y non nul arith} and in Corollary
\ref{Cor:inegality de pentes}. However, we
don't know whether in general case we can
remove the hypothesis that any epimorphism in
$\mathcal C$ has a kernel, although this
condition is fulfilled for all examples that
we have discussed in the previous section.

In the following, we give an example of
Harder-Narasimhan category, which will play
an important role in the next section. Let
$\mathcal C$ be an Abelian category and
$\mathcal E$ be the class of all short exact
sequences in $\mathcal C$. We suppose given a
rank function $\rang:K_0(\mathcal
C)\rightarrow\mathbb Z$. In this example we
take the totally ordered set $I$ as a subset
of $\mathbb R$ (with the induced order). For
any object $X$ in $\mathcal C$, let $A_0(X)$
be the set of isomorphism classes in
$\mathbf{Fil}_{X}^{I,\mathrm{self}}$. We have
shown in the previous section that $(\mathcal
C,\mathcal E,A_0)$ is an arithmetic exact
category. Any arithmetic object $\overline X=(X,h)$
of this arithmetic exact category may be
considered, after choosing a representative in
$h$, as an object $X$ in $\mathcal C$
equipped with an $\mathbb R$-filtration
$(X_\lambda)_{\lambda\in I}$ which is
separated, exhaustive, left continuous and of
finite length. We define a real
number\footnote{Here $\sup_\emptyset=0$ by
convention.}
\[\widehat{\deg}(\overline X)=
\sum_{\lambda\in I
}\lambda\Big(\rang(X_\lambda)-\sup_{
j>\lambda, j\in I }\rang(X_{j})\Big).\] The
summation above turns out to be finite since
the filtration is of finite length and its
value doesn't depend on the choice of the
representative in $h$. If $\overline
X=(X,(X_\lambda)_{\lambda\in I})$ and
$\overline Y=(Y,(Y_\lambda)_{\lambda\in I})$
are two arithmetic objects and if
$f:X\rightarrow Y$ is an isomorphism which is
compatible with arithmetic structures, then
for any $\lambda\in I$, we have
$\rang(X_\lambda)\le\rang(Y_\lambda)$.
Therefore we have $\widehat{\deg}(\overline X
)\le\widehat{\deg}(\overline Y)$ by Abel's
summation formula.

We now show that the function
$\widehat{\deg}$ defined above extends
naturally to a homomorphism from
$K_0(\mathcal C,\mathcal E,A_0)$ to $\mathbb
R$. Let
\[\xymatrix{\relax 0\ar[r]&X'\ar[r]^-{u}&X\ar[r]^-{p}&X''\ar[r]&0}\]
be a short exact sequence in $\mathcal C$.
Suppose that $\mathcal
F'=(X'_{\lambda})_{\lambda\in I}$ (resp.
$\mathcal F=(X_\lambda)_{\lambda\in I}$,
$\mathcal F''=(X_{\lambda}'')_{\lambda\in
I}$) is an $\mathbb R$-filtration of $X'$
(resp. $X$, $X''$) which is separated,
exhaustive, left continuous and of finite
length, and such that $\mathcal
F'=u^*(\mathcal F)$, $\mathcal
F''=p_*(\mathcal F)$. Then for any real
number $\lambda\in I$ we have a canonical
exact sequence
\[\xymatrix{0\ar[r]&X_\lambda'\ar[r]&
X_\lambda\ar[r]&X_\lambda''\ar[r]&0}.\]
Therefore, $\widehat{\deg}(X,[\mathcal
F])=\widehat{\deg}(X',[\mathcal
F'])+\widehat{\deg}(X'',[\mathcal F''])$.
Notice that an non-zero arithmetic object
$\overline X=(X,[\mathcal F])$ is semistable
if and only if the filtration $\mathcal F$
has a jumping set which reduces to a one
point set. If $\overline X$ is semistable
and if $\{\lambda\}$ is a jumping set of
$\mathcal F$, then the arithmetic slope of
$\overline X$ is just $\lambda$. Therefore,
if $\overline X= (X,[\mathcal F])$ and
$\overline Y=(Y,[\mathcal G])$ are two
semistable arithmetic objects such that
$\lambda:=\widehat{\mu}(\overline
X)>\widehat{\mu}(\overline Y)$, then any
morphism $f:X\rightarrow Y$ which is
compatible with filtrations sends $\mathcal
F(\lambda)=X$ into $\mathcal G(\lambda)=0$,
therefore is the zero morphism.

If $\overline X=(X,[\mathcal F])$ is a
non-zero arithmetic object, we denote by
$X_{\des}$ the non-zero object in the
filtration $\mathcal F$ having the maximal
index. The existence of $X_{\des}$ is
justified by the finiteness and the left
continuity of $\mathcal F$. The arithmetic
subobject $\overline X_{\des}$ of $\overline
X$ is semistable. Furthermore, for any
non-zero arithmetic subobject $\overline
Y=(Y,[\mathcal G])$ of $\overline X$, we have
\[\begin{split}\widehat{\mu}(\overline{Y})
&=\frac{1}{\rang(Y)} \sum_{\lambda\in I
}\lambda\Big(\rang(\mathcal
G(\lambda))-\sup_{j>\lambda,j\in I
}\rang(\mathcal G(j))\Big)\\&\le
\frac{1}{\rang(Y)}\sum_{\lambda\in I
}\widehat{\mu}(\overline
X_{\des})\Big(\rang(\mathcal
G(\lambda))-\sup_{j>\lambda,j\in I
}\rang(\mathcal
G(j))\Big)=\widehat{\mu}(\overline
X_{\des}).\end{split}\] The equality holds if
and only if $\overline Y$ is semistable and
of slope $\widehat{\mu}(\overline X_{\des})$,
in this case, the canonical morphism from $Y$
to $X$ factorizes through $X_{\des}$ since it
is compatible with filtrations. Hence we have
proved that $(\mathcal C,\mathcal
E,A_0,\widehat{\deg},\rang)$ is a
Harder-Narasimhan category.

Suppose that $\overline X=(X,[\mathcal F])$
is a non-zero arithmetic object, where
$\mathcal F=(X_\lambda)_{\lambda\in\mathbb
R}$. If
$E=\{\lambda_1>\lambda_2>\cdots>\lambda_n\}$
is the minimal jumping set of $\mathcal F$
(i.e. the intersection of all jumping sets of
$\mathcal F$, which is itself a jumping set
of $\mathcal F$), then
\[\xymatrix{0\ar[r]&X_{\lambda_1}\ar[r]&
X_{\lambda_2}\ar[r]&\cdots\ar[r]&X_{\lambda_n}=X}\]
is the Harder-Narasimhan sequence of
$\overline X$. Furthermore,
$\widehat{\mu}(\overline
X_{\lambda_1})=\lambda_1$, and for any $2\le
i\le n$, $\widehat{\mu}(\overline{
X_{\lambda_{i}}/
X_{\lambda_{i-1}}})=\lambda_i$.

\section{Harder-Narasimhan filtrations and polygons}

\hskip\parindent We fix in this section a
Harder-Narasimhan category $(\mathcal
C,\mathcal E,A,\widehat{\deg},\rang)$. We
shall introduce the notions of
Harder-Narasimhan filtrations and
Harder-Narasimhan measures for an arithmetic
object in $(\mathcal C,\mathcal
E,A,\widehat{\deg},\rang)$. We shall also
explain that if $\mathcal D$ is an Abelian
category equipped with a rank function and if
there exists an exact functor $F:\mathcal
C\rightarrow\mathcal D$ which preserves rank
functions, then for any non-zero arithmetic
object $\overline X$ in $\mathcal C$, the
Harder-Narasimhan filtration of $\overline X$
induces a filtration of $F(X)$, which defines
an arithmetic object $\overline{F(X)}$ of the
Harder-Narasimhan category defined by
$\mathbb R$-filtrations in $\mathcal D$ which
are separated, exhaustive, left continuous
and of finite length. Furthermore, the
Harder-Narasimhan polygon (resp. measure) of
$\overline{F(X)}$ coincides with that of
$\overline X$. Therefore, filtered objects in
Abelian categories equipped with rank
functions can be considered in some sense as
models to study Harder-Narasimhan polygons.

\begin{proposition}\label{Pro:Construction de la filtration de Harder-Narasimhan arith}
Let $\overline X$ be a non-zero arithmetic
object and
\[\xymatrix{0=X_0^{\HN}\ar[r]&X_1^{\HN}
\ar[r]&\cdots\ar[r]& X_{n-1}^{\HN}\ar[r]&
X_n^{\HN}=X}\] be its Harder-Narasimhan
sequence. If for any real number $\lambda$ we
denote by\footnote{By convention
$\max\emptyset=0$.}
\[i_{\overline X}(\lambda)=\max\{1\le i\le n\;|\;
\widehat{\mu}(\overline X_i^{\HN}/\overline
X_{i-1}^{\HN})\ge\lambda\}\] and $ X_\lambda=
X_{i_{\overline X}(\lambda)}^{\HN}$, then
$(X_\lambda)_{\lambda\in\mathbb R}$ is an
$\mathbb R$-filtration of the object $X$ in
$\mathcal C$. Furthermore, this filtration is
separated, exhaustive, left continuous and of
finite length.
\end{proposition}
\begin{proof}
If $\lambda>\lambda'$, then $i_{\overline
X}(\lambda)\le i_{\overline X}(\lambda')$,
hence $({X}_\lambda)_{\lambda\in\mathbb R}$
is an $\mathbb R$-filtration of $X$.
Moreover, for any $\lambda\in\mathbb R$,
$X_\lambda\in\{ X_0^{\HN},\cdots,
X_n^{\HN}\}$, therefore this filtration is of
finite length. When
$\lambda>\widehat{\mu}_{\max}(\overline X)$, we have
$i_{\overline X}(\lambda)=0$, which implies that $\overline
X_{\lambda}=\overline X_0^{\HN}=0$ is the
zero object, so the filtration is separated.
When $\lambda<\widehat{\mu}_{\min}(\overline
X)$, $i_{\overline X}(\lambda)=n$, so
$\overline X_{\lambda}=\overline X$, i.e.,
the filtration is exhaustive. To prove the
left continuity of this filtration, it
suffices to verify that the function
$\lambda\mapsto i_{\overline X}(\lambda)$ is
left continuous. Actually, this function is
left locally constant: if $i_{\overline
X}(\lambda)=0$, then for any integer $1\le
i\le n$, we have $\widehat{\mu}(\overline
X_i^{\HN}/\overline X_{i-1}^{\HN})<\lambda$,
so there exists $\varepsilon_0>0$ such that
for any $0\le\varepsilon<\varepsilon_0$, we
have $\widehat{\mu}(\overline
X_i^{\HN}/\overline
X_{i-1}^{\HN})<\lambda-\varepsilon$, i.e.,
$i_{\overline X}(\lambda-\varepsilon)=0$; if
$i_{\overline X}(\lambda)=n$, then for any
integer $1\le i\le n$ and any real number
$\varepsilon\ge 0$, we have
$\widehat{\mu}(\overline X_i^{\HN}/\overline
X_{i-1}^{\HN})\ge\lambda\ge\lambda-\varepsilon$,
so $i_{\overline X}(\lambda-\varepsilon)=n$;
finally if $1\le i_{\overline X}(\lambda)\le
n-1$, then we have $\widehat{\mu}(\overline
X_{i_{\overline X}(\lambda)}^{\HN}/\overline
X_{i_{\overline
X}(\lambda)-1}^{\HN})\ge\lambda$ and
$\widehat{\mu}(\overline X_{i_{\overline
X}(\lambda)+1}^{\HN}/\overline
X_{i_{\overline X}(\lambda)}^{\HN})<\lambda$,
hence there exists $\varepsilon_0>0$ such
that, for any $0\le
\varepsilon<\varepsilon_0$, we have
$\widehat{\mu}(\overline X_{i_{\overline
X}(\lambda)}^{\HN}/\overline X_{i_{\overline
X}(\lambda)-1}^{\HN})\ge\lambda-\varepsilon$
and $\widehat{\mu}(\overline X_{i_{\overline
X}(\lambda)+1}^{\HN}/\overline
X_{i_{\overline
X}(\lambda)}^{\HN})<\lambda-\varepsilon$,
i.e., $i_{\overline
X}(\lambda-\varepsilon)=i_{\overline
X}(\lambda)$.
\end{proof}

\begin{definition}
With the notations of Proposition
\ref{Pro:Construction de la filtration de
Harder-Narasimhan arith}, the filtration
$(X_\lambda)_{\lambda\in\mathbb R}$ is called
the {\it Harder-Narasimhan filtration } (or
{\it canonical filtration}) of $\overline X$,
denoted by $\HN(\overline X)$. Clearly,
$\widehat{\mu}_{\min}(\overline X_\lambda)\ge
\lambda$ for any $\lambda\in\mathbb R$. We
define the {\it Harder-Narasimhan filtration}
(or {\it canonical filtration}) of the zero
object to be its only $\mathbb R$-filtration
which associates to each $\lambda\in\mathbb
R$ the zero object itself.
\end{definition}

\begin{theorem} Keep the notations of
Proposition \ref{Pro:Construction de la
filtration de Harder-Narasimhan arith}.
Suppose in addition that any epimorphism in $\mathcal C$ has a
kernel in the case where $A$ is non-trivial. Then any morphism in $\mathcal C_A$
is compatible with Harder-Narasimhan
filtrations.\label{Thm:compatibilite avec la
filtration HN arith}
\end{theorem}
\begin{proof}
Let $f:X\rightarrow Y$ be a morphism which is
compatible with arithmetic structures. The
case where $X$ or $Y$ is zero is trivial. We
now suppose that $X$ and $Y$ are non-zero.
Let
\[\xymatrix{0=X_0^{\HN}\ar[r]&X_1^{\HN}\ar[r]
&\cdots\ar[r]&X_{n-1}^{\HN}\ar[r]&
X_n^{\HN}=X}\] be the Harder-Narasimhan
sequence of $\overline X$ and
\[\xymatrix{0=Y_0^{\HN}\ar[r]&Y_1^{\HN}
\ar[r]&\cdots\ar[r]&Y_{m-1}^{\HN}\ar[r]&
Y_m^{\HN}=Y}\] be the Harder-Narasimhan sequence of $Y$. For all
integers $0\le i<j\le m$, let $P_{j,i}$ be
the canonical morphism from $Y_j^{\HN}$ to
$Y^{\HN}_j/Y^{\HN}_i$. For any integer $0\le
i\le n$, let $U_i$ be the canonical
monomorphism from $X^{\HN}_i$ to $X$. Suppose
that $\lambda$ is a real number. If
$i_{\overline X}(\lambda)=0$ or if
$i_{\overline Y}(\lambda)=0$, we define
$F_\lambda$ as the zero morphism from
$X_\lambda$ to $Y_\lambda$; if $i_{\overline
Y}(\lambda)=m$, we have $\overline
Y_\lambda=\overline Y$ and we define
$F_\lambda$ as the composition
$fU_{i_X(\lambda)}$; otherwise we have
$\widehat{\mu}(\overline X_{i_{\overline
X}(\lambda)}^{\HN}/\overline
X^{\HN}_{i_{\overline
X}(\lambda)-1})\ge\lambda$ and
$\widehat{\mu}(\overline
Y^{\HN}_{i_{\overline Y}(\lambda)}/\overline
Y^{\HN}_{i_{\overline
Y}(\lambda)-1})\ge\lambda$, but
$\widehat{\mu}(\overline
Y^{\HN}_{j}/\overline Y^{\HN}_{j-1})<\lambda$
for any $j>i_{\overline Y}(\lambda)$. We will
prove by induction that the morphism
$fU_{i_{\overline X}(\lambda)}$ factorizes
through $Y^{\HN}_{i_{\overline Y}(\lambda)}$.
First it is obvious that the morphism
$fU_{i_{\overline X}(\lambda)}$ factorizes
through $Y^{\HN}_m=Y$. If it factorizes
through certain
$\varphi_j:X^{\HN}_{i_{\overline
X}(\lambda)}\rightarrow Y^{\HN}_j$, where
$j>i_{\overline Y}(\lambda)$, then the
composition $P_{j,j-1}\varphi_j$ must be zero
since (see Proposition \ref{Pro:mu max X ge
mu min Y si X vers Y non nul arith} and the
remark after its proof)
\[\widehat{\mu}(\overline
Y^{\HN}_j/\overline
Y^{\HN}_{j-1})<\lambda\le\widehat{\mu}(\overline
X^{\HN}_{i_{\overline X}(\lambda)}/\overline
X^{\HN}_{i_{\overline
X}(\lambda)-1})=\widehat{\mu}_{\min}(\overline
X^{\HN}_{i_{\overline X}(\lambda)}).
\]
So the morphism $fU_{i_{\overline
X}(\lambda)}$ factorizes through
$Y_{j-1}^{\HN}$. By induction we obtain that
$fU_{i_{\overline X}(\lambda)}$ factorizes
(in unique way) through a morphism
$F_\lambda:X_{i_{\overline
X}(\lambda)}\rightarrow Y_{i_{\overline
Y}(\lambda)}$. The family of morphisms
$F=(F_\lambda)_{\lambda\in\mathbb R}$ defines
a natural transformation such that $(F,f)$ is
a morphism of filtrations. Therefore the
morphism $f$ is compatible with
Harder-Narasimhan filtrations.
\end{proof}

\begin{remark}
Theorem \ref{Thm:compatibilite avec la
filtration HN arith} implies that $\HN$ defines
actually a functor from the category
$\mathcal C_A$ to the full sub-category
$\Fil^{\mathbb R,\mathrm{self}}(\mathcal C)$
of $\Fil^{\mathbb R}(\mathcal C)$ consisting
of $\mathbb R$-filtrations which are
separated, exhaustive, left continuous and of
finite length, which sends an arithmetic
object $\overline X$ to its Harder-Narasimhan
filtration.
\end{remark}

\begin{corollary} Suppose in the case where
$A$ is non-trivial that any epimorphism in
$\mathcal C$ has a kernel. Let $\overline X$
and $\overline Y$ be two arithmetic objects
and $f:Y\rightarrow X$ be a morphism which is
compatible with arithmetic structures. If
$\widehat{\mu}_{\min}(\overline
Y)\ge\lambda$, then the morphism $f$
factorizes through $X_{\lambda}$.
\end{corollary}
\begin{proof}
Since $f$ is compatible with arithmetic
structures, it is compatible with
Harder-Narasimhan filtrations. So the
restriction of $f$ on $Y_{\lambda}$
factorizes through $X_\lambda$. As
$\widehat{\mu}_{\min}(\overline
Y)\ge\lambda$, we have $Y_\lambda=Y$,
therefore $f$ factorizes through $X_\lambda$.
\end{proof}

Let $\overline X$ be a non-zero arithmetic
object and
\[\xymatrix{0=
X_0^{\HN}\ar[r]& X_1^{\HN}
\ar[r]&\cdots\ar[r]& X_{n-1}^{\HN}\ar[r]&
X_n^{\HN}=X}\] be its Harder-Narasimhan
sequence. For any integer $0\le i\le n$, we
note $t_i=\rang X_i^{\HN}/\rang X$. For any
integer $1\le i\le n$, we note
${\lambda}_i=\widehat{\mu}(X_{i}^{\HN}/X_{i-1}^{\HN})$.
Then the function
\[P_{\overline X}(t)=\sum_{i=1}^{n}
\left(\frac{\widehat{\deg} (\overline
X^{\HN}_{i-1})}{\rang X}+
{\lambda}_i(t-t_{i-1})\right)\indic_{[t_{i-1},
t_{i}]}(t)\] is a polygon\footnote{Namely a
concave function having value $0$ at the origin and
which is piecewise linear.} on $[0,1]$,
called the {\it normalized Harder-Narasimhan
polygon} of $\overline X$. The function
$P_{\overline X}$ takes value $0$ at the
origin, and its first order derivative is
given by
\[P_{\overline X}'(t)=\sum_{i=1}^n\lambda_i
\indic_{[t_{i-1},t_i[}(t).\] The probability
measure
\[\nu_{\overline X}:=\sum_{i=1}^n
\frac{\rang(X_{i}^{\HN})-\rang(
X_{i-1}^{\HN})}{\rang
X}\delta_{{\lambda}_i}=\sum_{i=1}^n(t_i-t_{i-1})
\delta_{\lambda_i}\] is called the {\it
Harder-Narasimhan measure} of $\overline X$.
We define the Harder-Narasimhan measure of
the zero arithmetic object to be the zero
measure on $\mathbb R$. After Proposition
\ref{Pro:Construction de la filtration de
Harder-Narasimhan arith}, if $\overline X$ is
a non-zero arithmetic object and if
$(X_\lambda)_{\lambda\in\mathbb R}$ is the
Harder-Narasimhan filtration of $\overline
X$, then the Harder-Narasimhan measure
$\nu_{\overline X}$ of $\overline X$ is the
first order derivative (in distribution
sense) of the function $t\longmapsto
-\rang(X_t)$. Finally we point out that the
Harder-Narasimhan polygon of a non-zero
arithmetic object $\overline X$ can be
uniquely determined in an explicit way from
its Harder-Narasimhan measure.

\begin{proposition}\label{Pro:comparaison de
slpes} Suppose in the case where $A$ is
non-trivial that any epimorphism in $\mathcal
C$ has a kernel. If $\overline X$ and
$\overline Y$ are two non-zero arithmetic
objects and if $f:X\rightarrow Y$ is an
isomorphism which is compatible to arithmetic
structures, then $\widehat{\mu}(\overline
X)\le\widehat{\mu}(\overline Y)$, and
therefore $\widehat{\deg}(\overline
X)\le\widehat{\deg}(\overline Y)$.
\end{proposition}
\begin{proof}
Let $(X_\lambda)_{\lambda\in\mathbb R}$ and
$(Y_\lambda)_{\lambda\in\mathbb R}$ be the
Harder-Narasimhan filtrations of $\overline X$
and of $\overline Y$ respectively. Theorem
\ref{Thm:compatibilite avec la filtration HN
arith} implies that $f$ is compatible with
filtrations. Hence
$\rang(X_\lambda)\le\rang(Y_\lambda)$ for any
$\lambda\in\mathbb R$. Therefore, by taking
an interval $[-M,M]$ containing
$\supp(\nu_{\overline
X})\cup\supp(\nu_{\overline Y})$, we obtain
\[\begin{split}
\widehat{\mu}(\overline
X)&=\int_{-M}^Mt\;\mathrm{d}\nu_{\overline X}
(t)=-\int_{-M}^Mt\;\mathrm{d}\rang(X_t)=\Big[-t\rang(X_t)
\Big]_{-M}^M+\int_{-M}^M\rang(X_t)\mathrm{d}t\\
&\le M\rang(X_M)+\int_{-M}^M\rang(Y_t)\mathrm{d}t=
M\rang(Y_M)+\int_{-M}^M\rang(Y_t)\mathrm{d}t=\widehat{\mu}(\overline
Y ).
\end{split}\]
\end{proof}

Let $\mathcal D$ be an Abelian category and
$\rang$ be a rank function on $\mathcal D$.
It is interesting to calculate explicitly the
Harder-Narasimhan filtration of an object $Y$
in $\mathcal D$, equipped with an $\mathbb
R$-filtration $\mathcal
F=(Y_\lambda)_{\lambda\in\mathbb R}$ which is
separated, exhaustive, left continuous and of
finite length. Let
$U=\{\lambda_1>\cdots>\lambda_n\}$ be the
minimal jumping set of the filtration
$\mathcal F$, then
\[\xymatrix{0\ar[r]&Y_{\lambda_1}\ar[r]&
Y_{\lambda_2}\ar[r]&\cdots\ar[r]&Y_{\lambda_n}=Y}\]
is the Harder-Narasimhan sequence of
$\overline Y=(Y,[\mathcal F])$. Therefore,
the Harder-Narasimhan filtration of
$\overline Y$ is just the filtration
$\mathcal F$ itself. So we have
\[P_{\overline Y}'(t)=\sum_{i=1}^n\lambda_i\indic_{[t_{i-1},t_i[}\]
where $t_0=0$, and for any $1\le i\le n$,
$t_i=\rang(Y_{\lambda_i})/\rang(Y)$. Furthermore,
\[\nu_{\overline Y}=\sum_{i=1}^n(t_i-t_{i-1})\delta_{\lambda_i}.\]

Let $F:\mathcal C\rightarrow\mathcal D$ be an
exact functor from $\mathcal C$ to an Abelian
category $\mathcal D$. The functor $F$
induces a functor $\widehat F:\mathcal
C_A\rightarrow\mathcal D$ which sends an
arithmetic object $\overline X$ to $F(X)$, it
also induces a homomorphism of groups
$K_0(F):K_0(\mathcal C,\mathcal
E,A)\rightarrow K_0(\mathcal D)$. Since $F$
is exact, it sends monomorphisms to
monomorphisms, therefore it induces a functor
$\widetilde F:\mathbf{Fil}^{\mathbb
R,\mathrm{self}}(\mathcal
C)\rightarrow\mathbf{Fil}^{\mathbb
R,\mathrm{self}}(\mathcal D)$. If $\overline
X$ is an arithmetic object of $(\mathcal
C,\mathcal E,A)$, then $\widetilde
F(\HN(\overline X))$ is an $\mathbb
R$-filtration of $F(X)$. The following
proposition shows that we can recover the
Harder-Narasimhan polygon and the
Harder-Narasimhan measure of $\overline X$
from the filtration $\widetilde
F(\HN(\overline X))$.

\begin{proposition}
\label{Pro:le polygone sur la fibre generique
egale au polygone initial} Suppose given a
rank function $\rang$ on $K_0(\mathcal D)$
(which defines a Harder-Narasimhan category
structure on $\mathcal D$) such that the
functor $F$ preserves rank functions (i.e.
$\rang(F(X))=\rang(X)$ for any
$X\in\obj\mathcal C$). Then for any
arithmetic object $\overline X$ in $\mathcal
C_A$, the normalized Harder-Narasimhan
polygon of the filtration
$\overline{F(X)}=(F(X),[\widetilde
F(\HN(\overline X))])$ coincides with that of
$\overline X$, and the Harder-Narasimhan
measure of $\overline{F(X)}$ coincides with
that of $\overline X$.
\end{proposition}
\begin{proof}
Since the Harder-Narasimhan filtration of
$\overline{F(X)}$ coincides with $\widetilde
F(\HN(X))$, the function $t\longmapsto
-\rang(\HN(\overline X)(t))$ identifies with
$t\longmapsto -\rang(\widetilde
F(\HN(\overline X))(t))$. Therefore
$\nu_{\overline{F(X)}}=\nu_{\overline X}$ and
hence $P_{\overline{F(X)}}=P_{\overline X}$.
\end{proof}

Let $(\mathcal C,\mathcal E,A)$ be an
arithmetic exact category, $\widehat{\deg}$
be a degree function on $(\mathcal C,\mathcal
E,A)$ and $\rang$ be a rank function on
$(\mathcal C,\mathcal E)$. If $(\mathcal
C,\mathcal E)$ is an Abelian category, then
the axioms for $(\mathcal C,\mathcal
E,A,\widehat{\deg},\rang)$ to be a
Harder-Narasimhan category can be
considerably simplified. We shall show this
fact in Proposition \ref{Pro:HN for Abelian
category}.

\begin{proposition}\label{Pro:HN for Abelian
category} Supose that $(\mathcal C,\mathcal E)$ is an Abelian category. Then $(\mathcal C,\mathcal
E,A,\widehat{\deg},\rang)$ is a
Harder-Narasimhan category if the following
conditions are satisfied:
\begin{enumerate}[1)]
\item for any non-zero arithmetic object
$\overline X$, there exists a non-zero
arithmetic subobject $\overline Z$ of
$\overline X$ such that
\begin{equation}\label{Equ:pentem aximal}\widehat{\mu}(\overline Z)=\sup\{ \widehat{\mu}(\overline
Y)\;|\; \overline Y \text{ is a non-zero
arithmetic subobject of }\overline X
\};\end{equation}
\item for any non-zero object $X$ in $\mathcal C$ and
for any two arithmetic structures $h_X$ and
$h_X'$ on $X$, if $\Id_X:(X,h_X)\rightarrow
(X,h_X')$ is compatible with arithmetic
structures, then
$\widehat{\mu}(X,h_X)\le\widehat{\mu}(X,h_X')$.
\end{enumerate}
\end{proposition}

Note that the condition 1) is verified once $\{ \widehat{\mu}(\overline
Y)\;|\; \overline Y \text{ is a non-zero
arithmetic subobject of }\overline X
\}$ is a finite set, or equivalently $\{ \widehat{\deg}(\overline
Y)\;|\; \overline Y \text{ is a non-zero
arithmetic subobject of }\overline X
\}$ is a finite set for any non-zero arithmetic object $\overline X$.

The following technical lemma, which is dual
to Lemma \ref{Lem:compose un epi est
compatible implie compatible}, is useful for
the proof of Proposition \ref{Pro:HN for
Abelian category}.

\begin{lemma}\label{Lem:compose un mono est
compatible implie compatible} Let $(\mathcal
C,\mathcal E,A)$ be an arithmetic exact
category. Suppose that any monomorphism in
$\mathcal C$ has a cokernel. Let $(X,h_X)$
and $(Y,h_Y)$ be two arithmetic objects and
$f:X\rightarrow Y$ be a morphism in $\mathcal
C$. Suppose that $(Y,h_Y)$ is an arithmetic
subobject of an arithmetic object $(Z,h_Z)$
and $u:Y\rightarrow Z$ is the inclusion
morphism. Then the morphism $f$ is compatible
with arithmetic structures if and only if it
is the case for $uf$.
\end{lemma}

{\noindent{\it Proof of Proposition
\ref{Pro:HN for Abelian category}.\ \ \ }}
Suppose that $\overline X_{\des}$ is a
non-zero arithmetic subobject of $\overline
X$ verifying \eqref{Equ:pentem aximal}, whose
rank $r$ is maximal. Suppose that $\overline
Z$ is another non-zero arithmetic subobject
of $\overline X$ verifying \eqref{Equ:pentem
aximal}. Consider the short exact sequence
\[\xymatrix{0\ar[r]&Z\cap X_{\des}\ar[r]&Z\oplus X_{\des}
\ar[r]&Z+X_{\des}\ar[r]&0},\] where $Z\cap
X_{\des}$ is the fiber product
$Z\times_XX_{\des}$ and $Z+X_{\des}$ is the
canonical image of $Z\oplus X_{\des}$ in $X$.
Therefore,
\[\widehat{\deg}(\overline{Z\cap X_{\des}})
+\widehat{\deg}(\overline{Z+X_{\des}})
=\widehat{\deg} (\overline
Z)+\widehat{\deg}(\overline
X_{\des})=\alpha(\rang(Z)+\rang(X_{\des})),\] so
\[\begin{split}\widehat{\deg}(\overline{Z+X_{\des}})&=
\alpha(\rang(Z)+\rang(X_{\des}))
-\widehat{\deg}(\overline{Z\cap
X_{\des}})\\&\ge
\alpha(\rang(Z)+\rang(X_{\des})-
\rang(\overline{Z\cap
X_{\des}}))=\alpha\rang(\overline{Z
+X_{\des}}),\end{split}\] which means that
$\widehat{\mu}(\overline{Z+X_{\des}})=\alpha$,
and hence $\rang(Z+X_{\des})=\rang(X_{\des})$
since $\rang(X_{\des})$ is maximal. As $\rang
$ is a rank function, we obtain $Z=X_{\des}$.
Therefore, the axiom ({\bf HN}\ref{Axiom:HNA
sous fibre destabilisant}) is fulfilled.

We now verify the axiom ({\bf
HN}\ref{Axiom:HNA comparaison de pente pour
des objet semistable}). Let $\overline
X=(X,h_X)$ and $\overline Y=(Y,h_Y)$ be two
semistable arithmetic objects. Suppose that
there exists a non-zero morphism
$f:X\rightarrow Y$ which is compatible with
arithmetic objects. Let $Z$ be the image of
$f$ in $Y$, $u:Z\rightarrow Y$ be the
canonical inclusion and $\pi:X\rightarrow Z$
be the canonical projection. The fact that
$f$ is compatible with arithmetic structures
implies that the identity morphism
$\Id_Z:(Z,\pi_*h_X)\rightarrow (Z,u^*h_Y)$ is
compatible with arithmetic structures (after
Lemmas \ref{Lem:compose un epi est compatible
implie compatible} and \ref{Lem:compose un
mono est compatible implie compatible}).
Therefore, the semistability of $\overline X$
and of $\overline Y$, combining the condition
2), implies that $\widehat{\mu}(\overline
X)\le\widehat{\mu}(Z, \pi_*h_X)\le\widehat{\mu}(Z,u^*h_Y)\le\widehat{\mu}(\overline Y)$.
\endzm

\begin{corollary}
\label{Cor:several aritmmetic sructio}
Let $(\mathcal C,\mathcal E)$ be an Abelian category equipped with a rank function $\rang$, $n\ge 2$ be an integer, $(A_i)_{1\le i\le n}$ be a family of arithmetic structures on $(\mathcal C,\mathcal E)$ and $A=A_1\times\cdots\times A_n$. Suppose given for any $1\le i\le n$ a degree function $\widehat{\deg}_i$ on $(\mathcal C,\mathcal E,A_i)$ such that
\begin{enumerate}[1)]
\item $\{\widehat{\deg}_i(\overline
Y)\;|\; \overline Y \text{ is a non-zero
arithmetic subobject of }\overline X
\}$ is a finite set for any non-zero arithmetic objec $\overline X$;
\item $(\mathcal C,\mathcal E,A_i,\widehat{\deg}_i,\rang)$ is a Harder-Narasimhan category.
\end{enumerate}
Then for any $\alpha=(a_i)_{1\le i\le n}\in\mathbb R_{\ge 0}^n$, if we denote by $\widehat{\deg}_\alpha=\sum_{i=1}^n a_i\widehat{\deg}_{i}$, then $(\mathcal C,\mathcal E,A,\widehat{\deg}_\alpha,\rang)$ is a Harder-Narasimhan category.
\end{corollary}

\section{Examples of Harder-Narasimhan categories}

\hskip\parindent In this section, we shall
give some example of Harder-Narasimhan
categories.

\subsection*{Filtrations in an extension of Abelian categories}

\hskip\parindent Let $\mathcal C$ and
$\mathcal C'$ be two Abelian categories and
$F:\mathcal C\rightarrow\mathcal C'$ be an
exact functor which sends a non-zero object
of $\mathcal C$ to a non-zero object of
$\mathcal C'$. Let $\mathcal E$ (resp.
$\mathcal E'$) be the class of all exact
sequences in $\mathcal C$ (resp. $\mathcal
C'$). Suppose given a rank function
$\rang':K_0(\mathcal C',\mathcal
E')\rightarrow\mathbb R$. Let $I$ be a
non-empty subset of $\mathbb R$, equipped
with the induced order. For any object $X$ in
$\mathcal C$, let $A(X)$ be the set of
isomorphism classes of objects in
$\mathbf{Fil}_{F(X)}^{I,\mathrm{self}}$.
Suppose that $h=[\mathcal F]$ is an element in $A(X)$.
For any monomorphism $u:X_0\rightarrow X$, we
define $u^*(h)$  to be the
class $[F(u)^*\mathcal F]\in A(X_0)$. For any epimorphism $p:X\rightarrow Y$, we define $p_*(h)$ to be
$[F(p)_*\mathcal F]\in A(Y)$. Similarly to
the the case of filtrations in an Abelian category, $(\mathcal C,\mathcal
E,A)$ is an arithmetic exact category. By
definition we know that if $\overline
X_i=(X_i,[\mathcal F_i])$ ($i=1,2$) are two
arithmetic objects, then a morphism
$f:X_1\rightarrow X_2$ in $\mathcal C$ is
compatible with arithmetic structures if and
only if $F(f)$ is compatible with filtrations
$(\mathcal F_1,\mathcal F_2)$. For any
arithmetic object $\overline X$ of $(\mathcal
C,\mathcal E,A)$, we define the arithmetic
degree of $\overline X=(X,[\mathcal F])$ to
be the real number
\[\widehat{\deg}(\overline X)=\sum_{\lambda\in I}\lambda\Big(
\rang'(\mathcal
F(\lambda))-\sup_{j>\lambda,j\in
I}\rang'(\mathcal F(j))\Big).\] Since $F$ is
an exact functor, $\widehat{\deg}$ extends
naturally to a homomorphism from
$K_0(\mathcal C,\mathcal E,A)$ to $\mathbb
R$. In the previous section we have shown
that if we define, for any object $X'$ in
$\mathcal C'$, $A'(X')$ as the set of all
isomorphism classes of objects in
$\mathbf{Fil}_{X'}^{I,\mathrm{self}}$, then
$(\mathcal C',\mathcal E',A')$ is an
arithmetic category. Furthermore, if for any
arithmetic object $\overline X'=(X',[\mathcal
F'])$, we define
\[\widehat{\deg}{}'(\overline X')=\sum_{\lambda\in I}\lambda\Big(
\rang'(\mathcal
F'(\lambda))-\sup_{j>\lambda,j\in I}
\rang'(\mathcal F'(j))\Big),\] then
$\widehat{\deg}{}'$ extends naturally to a
homomorphism $K_0(\mathcal C',\mathcal
E',A')\rightarrow\mathbb R$, and $(\mathcal
C',\mathcal E',A',\widehat{\deg}{}',\rang')$
is a Harder-Narasimhan category. Notice that
for any object $(X,[\mathcal F])$ in
$\mathcal C_A$, we have
\[\widehat{\deg}(X,[\mathcal F])=\widehat{\deg}{}'(F(X),[\mathcal F]).\]

\begin{proposition}
Denote by $\rang$ the composition
$\rang'\circ K_0(F)$. Then $(\mathcal
C,\mathcal E,A,\widehat{\deg},\rang)$ is a
Harder-Narasimhan category.
\end{proposition}
\begin{proof} Since $F$ is an exact functor
which sends non-zero objects to non-zero
objects, the homomorphism $\rang$ is a rank
function. Let $\overline X=(X,[\mathcal F])$
be a non-zero arithmetic object in $\mathcal
C_A$. First we show that
$S:=\big\{\widehat{\deg}(\overline
Y)\;|\;\overline Y\text{ is an arithmetic
subobject of }\overline X\big\}$ is a finite
set. Let $U=\{\lambda_1,\cdots,\lambda_n\}$
be a jumping set of $\mathcal F$. If
$u:Y\rightarrow X$ is a monomorphism, then
$U$ is also a jumping set of $F(u)^*\mathcal
F$, therefore,
\[\widehat{\deg}(Y,[F(u)^*\mathcal F])\in
\Big\{\sum_{i=1}^na_i\lambda_i\;\Big|\;
\forall 1\le i\le n,\;a_i\in\mathbb N,\;0\le
a_1+\cdots+a_n\le\rang(X)\Big\}.\] The latter
is clearly a finite set.
Therefore, the condition 1) of Proposition
\ref{Pro:HN for Abelian category} is
satisfied. If $X$ is an object in $\mathcal
C$ and if $\mathcal F$ and $\mathcal G$ are
two filtrations of $F(X)$ such that
$\Id_{F(X)}=F(\Id_X)$ is compatible to
filtrations $(\mathcal F,\mathcal G)$, then
after Proposition \ref{Pro:comparaison de
slpes}, $\widehat{\deg}{}'(F(X),[\mathcal
F])\le \widehat{\deg}{}'(F(X),[\mathcal G])$
and therefore $\widehat{\mu}(X,[\mathcal
F])\le\widehat{\mu}(X,[\mathcal G])$. After
Proposition \ref{Pro:HN for Abelian
category}, $(\mathcal C,\mathcal
E,A,\widehat{\deg},\rang)$ is a
Harder-Narasimhan category.
\end{proof}

\begin{remark}
By Corollary \ref{Cor:several aritmmetic sructio}, we can easily generalize the formalism of Harder and Narasimhan to the case of bojects in $\mathcal C$ equipped with several filtrations of their images by $F$ in $\mathcal C'$.
\end{remark}

\section*{Filtered $(\varphi,N)$-modules}

\hskip\parindent Let $K$ be a field of
characteristic $0$, equipped with a discrete
valuation $v$ such that $K$ is complete for
the topology defined by $v$. Suppose that the
residue field $k$ of $K$ is of characteristic
$p>0$. Let $K_0$ be the fraction field of
Witt vector ring $W(k)$ and
$\sigma:K_0\rightarrow K_0$ be the absolute
Frobenius endomorphism. We call {\it
$(\varphi,N)$-module} (see \cite{Fontaine94}, \cite{Totaro96}, and \cite{Colmez_Fontaine00} for details) any finite dimensional
vector space $D$ over $K_0$, equipped with
\begin{enumerate}[1)]
\item a bijective $\sigma$-linear
endomorphism $\varphi:D\rightarrow D$,
\item a $K_0$-linear endomorphism $N:D\rightarrow D$
such that $N\varphi=p\varphi N$.
\end{enumerate}
Let $\mathcal C$ be
the category of all $(\varphi,N)$-modules.
It's an Abelian category. We denote by
$\mathcal E$ the class of all short exact
sequences in $\mathcal C$. There exists a
natural rank function $\rang$ on the category
$\mathcal C$ defined by the rank of vector
space over $K_0$. Furthermore, we have an
exact functor $F$ from $\mathcal C$ to the
category $\mathbf{Vec}_K$ of all finite
dimensional vector spaces over $K$, which
sends a $(\varphi,N)$-module $D$ to
$D\otimes_{K_0}K$. Consider the arithmetic
structure $A$ on $(\mathcal C,\mathcal E)$
such that, for any $(\varphi,N)$-module $D$,
$A(D)$ is the set of isomorphism classes of
$\mathbb Z$-filtrations of
$F(D)=D\otimes_{K_0}K$. Then $(\mathcal
C,\mathcal E,A)$ becomes an arithmetic exact
category. The objects in $\mathcal C_A$ are
called {\it filtered $(\varphi,N)$-modules}.

To each $(\varphi,N)$-module $D$ we associate
an integer
$\deg_{\varphi}(D)=-v(\det\varphi)$. If
$\overline D=(D,[\mathcal F])$ is a filtered
$(\varphi,N)$-module, we define
\[\deg_F(\overline D):=\sum_{i\in\mathbb Z}
i\Big(\rang_K(\mathcal F(i))-\rang_K(\mathcal
F(i+1))\Big)\qquad\text{and}\qquad
\widehat{\deg}(\overline D)=\deg_F(\overline
D)+\deg_{\varphi}(D).\] It is clear that
$\widehat{\deg}$ is a degree function on
$(\mathcal C,\mathcal E,A)$.

\begin{proposition}
$(\mathcal C,\mathcal
E,A,\widehat{\deg},\rang)$ is a
Harder-Narasimhan category.
\end{proposition}
\begin{proof}
Let $\overline X=(X,[\mathcal F])$ be a
non-zero filtered $(\varphi,N)$-module. We
have shown in the previous example that
$S_F=\{\deg_F(\overline Y)\;|\; \overline
Y\text{ is an arithmetic subobject of
}\overline X\}$ is a finite set. By the
isoclinic decomposition we obtain that
$S_\varphi=\{\deg_\varphi(Y)\;|\; Y\text{ is
a subobject of } X\}$ is also a finite set.
Therefore,
\[\widetilde{S}=\{\widehat{\mu}(\overline Y)\;|\;
\overline Y\text{ is an arithmetic subobject
of }\overline X\}\] is a finite set, and
hence the condition 1) of Proposition
\ref{Pro:HN for Abelian category} is
verified.

Suppose that $X$ is a $(\varphi,N)$-module
and $\mathcal F$ and $\mathcal G$ are two
$\mathbb Z$-filtrations of $X$ such that
$\Id_X$ is compatible with filtrations
$(\mathcal F,\mathcal G)$. We have shown in
the previous example that $\deg_F(X,\mathcal
F)\le\deg_F(X,\mathcal G)$. Hence
$\widehat{\deg}(X,\mathcal
F)\le\widehat{\deg}(X,\mathcal G)$.
Therefore, the condition 2) of Proposition
\ref{Pro:HN for Abelian category} is
verified, and hence $(\mathcal C,\mathcal
E,A,\widehat{\deg},\rang)$ is a
Harder-Narasimhan category.
\end{proof}

Note that semistable filtered
$(\varphi,N)$-modules having slope $0$ are
nothing but admissible filtered
$(\varphi,N)$-modules. In classical
literature, such filtered
$(\varphi,N)$-modules are said to be weakly
admissible. In fact, Colmez and Fontaine
\cite{Colmez_Fontaine00} have proved that all
weakly admissible $(\varphi,N)$-modules are
admissible, which had been a conjecture of
Fontaine.

\subsection*{Torsion free sheaves on a polarized projective variety}

\hskip\parindent Let $X$ be a geometrically
normal projective variety of dimension $d\ge
1$ over a field $K$ and $L$ be an ample
invertible $\mathcal O_X$-module. We denote
by $\mathbf{TF}(X)$ the category of torsion
free coherent sheaves on $X$. Notice that if
$\xymatrix{0\ar[r]&E'\ar[r]&E\ar[r]&E''\ar[r]&0}$
is an exact sequence of coherent $\mathcal
O_X$-modules such that $E'$ and $E''$ are
torsion free, then also is $E$. Therefore,
$\mathbf{TF}(X)$ is an exact sub-category of
the Abelian category of all coherent
$\mathcal O_X$-modules on $X$. Let $\mathcal
E$ be the class of all exact sequences in
$\mathbf{TF}(X)$ and let $A$ be the trivial
arithmetic structure on it. If $E$ is a
torsion free coherent $\mathcal O_K$-module,
we denote by $\rang(E)$ its rank and by
$\deg(E)$ the intersection number
$c_1(L)^{d-1}c_1(E)$. The mapping $\deg$
(resp. $\rang$) extends naturally to a
homomorphism from $K_0(\mathbf{TF}(X))$ to
$\mathbb R$ (resp. $\mathbb Z$). A classical
result \cite{Maruyama81} (see also
\cite{Shatz77}) shows that
$(\mathbf{TF}(X),\mathcal E,A,\deg,\rang)$ is
in fact a Harder-Narasimhan category.

\subsection*{Hermitian vector bundles on the spectrum of
an algebraic integer ring}

\hskip\parindent Let $K$ be a number field
and $\mathcal O_K$ be its integer ring. We
denote by $\mathbf{Pro}(\mathcal O_K)$ the
category of all projective $\mathcal
O_K$-modules of finite type. Let $\mathcal E$
be the family of short exact sequences of
projective $\mathcal O_K$-modules of finite
type. Then $(\mathbf{Pro}(\mathcal
O_K),\mathcal E)$ is an exact category.

We denote by $\Sigma_f$ the set of all finite
places of $K$ which identifies with the set of
closed points of $\Spec\mathcal O_K$. If
$\mathfrak p$ is an element in $\Sigma_f$, we
denote by $v_{\mathfrak
p}:K^\times\rightarrow\mathbb Z$ the
valuation associated to $\mathfrak p$ which
sends a non-zero element $a\in\mathcal O_K$
to the length of the Artinian local ring
$\mathcal O_{K,\mathfrak p}/a\mathcal
O_{K,\mathfrak p}$. Let $\mathbb F_{\mathfrak
p }:=\mathcal O_{K,\mathfrak p}/\mathfrak
p\mathcal O_{K,\mathfrak p}$ be the residue
field and $N_{\mathfrak p}$ be its cardinal.
We denote by $|\cdot|_{\mathfrak p}$ the
absolute value on $K$ such that
$|x|_{\mathfrak p}=N_{\mathfrak
p}^{-v_{\mathfrak p}(x)}$ for any $x\in
K^\times$. Let $\Sigma_\infty$ be the set of
all embeddings of $K$ in $\mathbb C$, whose
cardinal is $[K:\mathbb Q]$. For any
$\sigma\in\Sigma_\infty$, let
$|\cdot|_\sigma:K\rightarrow\mathbb R_{\ge
0}$ be the Archimedian absolute value such
that $|x|_\sigma=|\sigma(x)|$. The complex
conjugation defines an involution
$\sigma\mapsto\overline{\sigma}$ on
$\Sigma_\infty$. The product formula asserts
that for any $x\in K^\times$, $|x|_{\mathfrak
p }=1$ for almost all finite places
$\mathfrak p$, and we have
\[\prod_{\mathfrak p\in\Sigma_f}|x|_{\mathfrak p}
\prod_{\sigma\in\Sigma_\infty}|x|_{\sigma}=1.\]

Notice that a Hermitian vector bundle over
$\Spec\mathcal O_K$ is nothing other than a
pair $\overline E=
(E,(\|\cdot\|_\sigma)_{\sigma\in\Sigma_\infty})$,
where $E$ is a projective $\mathcal
O_K$-module of finite type $E$, and for any
$\sigma\in\Sigma_\infty$,
$\|\cdot\|_{\sigma}$ is a Hermitian metric on
$E\otimes_{\mathcal O_K,\sigma}\mathbb C$
such that $\|x\otimes
z\|_\sigma=\|x\otimes\overline{
z}\|_{\overline{\sigma}}$. The rank of the
Hermitian vector bundle $\overline E$ is just
defined to be that of $E$. The rank function
on $\mathbf{Pro}(\mathcal O_K)$ extends
naturally to a homomorphism from
$K_0(\mathbf{Pro}(\mathcal O_K))$ to $\mathbb
Z $. If $\overline E$ is a Hermitian vector
bundle of rank $r$, the (normalized) Arakelov
degree of $\overline E$ is by definition
\[\widehat{\deg}_n\overline E=\frac{1}{[K:\mathbb Q]}
\Big(\log\#(E/\mathcal O_Ks_1+\cdots+\mathcal
O_Ks_r )-\frac
12\sum_{\sigma\in\Sigma_\infty}\log\det(
\left<s_i,s_j\right>_\sigma)\Big),\] where
$(s_1,\cdots,s_r)\in E^r$ is an arbitrary
element in $E^r$ which defines a basis of
$E_K$ over $K$. This definition doesn't
depend on the choice of $(s_1,\cdots,s_r)$.
For more details, see
\cite{Bost2001} and \cite{Chambert}.

If for any projective $\mathcal O_K$-module
of finite type $E$, we denote by $A(E)$ the
set of all Hermitian structures on $E$, then
$(\mathbf{Pro}(\mathcal O_K),\mathcal E,A)$
is an arithmetic exact category, as we have
shown in the previous section. The category
$\mathbf{Pro}(\mathcal O_K)_A$ is the
category of all Hermitian vector bundles over
$\Spec\mathcal O_K$ and all homomorphism of
$\mathcal O_K$-modules having norm $\le 1$ at
every $\sigma\in\Sigma_\infty$. Furthermore,
if $\xymatrix{0\ar[r]&\overline
E'\ar[r]&\overline E\ar[r]&\overline
E''\ar[r]&0}$ is a sequence in $\mathcal
E_A$, then we have the equality
$\widehat{\deg}_n(\overline
E)=\widehat{\deg}_n(\overline
E')+\widehat{\deg}_n(\overline E'')$.
Therefore, $\widehat{\deg}_n$ extends to a
homomorphism from $K_0(\mathbf{Pro}(\mathcal
O_K),\mathcal E,A)$ to $\mathbb R$. The
results of Stuhler \cite{Stuhler76} and
Grayson \cite{Grayson76} show that
$(\mathbf{Pro}(\mathcal O_K),\mathcal E,A,
\widehat{\deg}_n,\rang)$ is a
Harder-Narasimhan category.

A recent work of Moriwaki \cite{Moriwaki06} generalizes the notion of semistability and Harder-Narasimhan flag to Hermitian torsion free coherent sheaves on normal arithmetic varieties. His appoach may also be adapted into the framework of Harder-Narasimhan categories.

\bibliography{chen}
\bibliographystyle{plain}
\vspace{1cm}
CMLS Ecole Polytechnique, Palaiseau 91120,
France (huayi.chen@polytechnique.org)
\end{document}